\documentclass[10pt]{amsart}
\usepackage{amsmath,amssymb,amsthm}
\newtheorem{thm}{Theorem}[section]
\newtheorem{lem}[thm]{Lemma}
\newtheorem{prop}[thm]{Proposition}
\newtheorem{cor}[thm]{Corollary}
\numberwithin{equation}{section}

\theoremstyle{example}

\theoremstyle{definition}

\theoremstyle{notation}

\theoremstyle{remark}
\newtheorem{remark}{Remark}[section]

\newcommand{\C}{{\mathbb C}}
\newcommand{\N}{{\mathbb N}}
\newcommand{\Z}{{\mathbb Z}}
\newcommand{\R}{{\mathbb R}}
\newcommand{\T}{{\mathbb T}}

\def\11{{\rm 1~\hspace{-1.4ex}l} }

\title
[Invariant measures for BO ]
{Invariant measures and long time behaviour for the  Benjamin-Ono equation II}
\author[Nikolay Tzvetkov]{Nikolay~Tzvetkov}
\author[Nicola Visciglia]{Nicola Visciglia}
\address{D\'epartement de Math\'ematiques, Universit\'e de Cergy-Pontoise, 2, 
avenue Adolphe Chauvin, 95302 Cergy-Pontoise  
Cedex, France and Institut Universitaire de France}\email{nikolay.tzvetkov@u-cergy.fr}
\address{Universit\`a Degli Studi di Pisa, Dipartimento di Matematica,
Largo Bruno Pontecorvo 5 I - 56127 Pisa. Italy}\email{ viscigli@dm.unipi.it}
\begin{document}
\maketitle

\begin{abstract}
As a continuation of our previous work on the subject, we prove new measure invariance results for the Benjamin-Ono equation.
The measures are associated with conservation laws whose leading term is a fractional Sobolev norm of order larger or equal than $5/2$. 
The new ingredient, compared with the case of conservation laws 
whose leading term is an integer Sobolev norm of order larger or equal than $3$
(that has been studied in our previous work), is the use of suitable orthogonality relations satisfied
by multilinear products of centered complex independent gaussian variables.
We also give some partial results for the measures associated with the two remaining conservation laws at lower regularity.
We plan to complete the proof of their invariance in a separated article which will be the final in the series. 
Finally in an appendix, we give a brief comparing of the recurrence properties of the flows of  Benjamin-Ono and KdV equations.
\end{abstract}
%%%%%%%%%%%%%%%%%%%%
\section{Introduction}
In \cite{TV} we constructed an infinite sequence of weighted gaussian measures associated with each conservation law of the Benjamin-Ono equation. The proof of the invariance of these measures under the flow of the equation turned out to be a quite delicate problem. In \cite{TV2}, we introduced an argument
which allowed to prove the invariance of some of the measures constructed in \cite{TV}. This argument is less dependent of the particular behaviour of each trajectory compared to previous works on the subject (\cite{B,B2,BT2,BTT,NORS,OH,tz_fourier,zh} ...).  
In particular, it reduces the matters to a verification of an asymptotic average property of the initial distribution of approximated problems. The verification of this property turned out to be an intricate problem. In \cite{TV2}, we succeed to solve it for a part of the Benjamin-Ono conservation laws, by exploiting the fine algebraic structure of the Benjamin-Ono conservation laws. The approach in \cite{TV2} is not applicable for the remaining cases. In the present paper we will deal with a large part of the remaining conservation laws. The new ingredient with respect to \cite{TV2} will be the use of the random oscillations on each mode to get an orthogonality which will allow us to get key  asymptotic average property.

Consider thus the Cauchy problem for Benjamin-Ono equation (in the periodic setting)
\begin{equation}\label{bo}
\partial_t u + H \partial_x^2 u+ u\partial_x u=0,\quad u(0)=u_0, \,\, (t, x)\in \R\times {\R}/2\pi \Z\,.
\end{equation}
From now on we shall always assume that $\dot H^s$
and $H^s$ are the Sobolev spaces defined on the one dimensional torus.
Since the mean value is conserved under the flow of \eqref{bo}, we can assume that the zero Fourier coefficient of the solutions of \eqref{bo} is zero and for
such functions $\dot H^s$ and $H^s$ norms are equivalent.
Thanks to \cite{M}, the problem \eqref{bo} is globally well-posed in the Sobolev spaces $H^s$, $s\geq 0$. 
We refer to \cite{ABFS, P, MST, BP,IK,M,T,KK, MP} for previous result on the Cauchy problem for the Benjamin-Ono equation.
We also quote the remarkable recent paper by Deng \cite{D}, 
where the Cauchy problem has been studied in functional spaces larger than $L^2$. These spaces are large enough, so that one can show the
invariance of the measure associated with the energy conservation law, constructed in \cite{tz}. The energy conservation law is also conserved by the approximated
problems and thus in \cite{D} one does not need to resolve the difficulty we face here and in \cite{TV2}. 

We note by $\Phi_t$, $t\in\R$ the flow established in \cite{M} and 
for every subset $A\subset H^s$ (with $s\geq 0$) and for every $t\in \R$ we
define the set $\Phi_t(A)$ as follows:
\begin{equation}\label{eq:PhiBO}\Phi_t(A)=\{u(t,.)\in H^s|  \hbox{ where }
u(t,.) \hbox{ solves } \eqref{bo} \hbox{ with } u_0 \in A\}.
\end{equation}
%%%%%%%%%%%%%%%%%%%%%%%%%
We now recall some notations from our previous paper \cite{TV}.
Smooth solutions to \eqref{bo} satisfy infinitely many conservation laws (see e.g. \cite{Mats}).
More precisely for $k\geq 0$ an integer, there is a conservation law of \eqref{bo} of the form 
\begin{equation}\label{strucRmezzi}
E_{k/2}(u)= \|u\|_{\dot{H}^{k/2}}^2 + R_{k/2} (u) 
\end{equation}
where all the terms that appear in
$R_{k/2}$  are homogeneous in $u$ of order  larger or equal to three and contain less than $k$ total number of derivatives.
Denote by $\mu_{k/2}$ the gaussian measure induced 
by the random Fourier series 
\begin{equation}\label{randomized}
\varphi_{k/2}(x, \omega)=\sum_{n\in \Z \setminus \{ 0\}} 
\frac{g_n(\omega)}{|n|^{k/2}} e^{{\bf i}nx}.
\end{equation}
In \eqref{randomized}, $(g_{n}(\omega))$ is a sequence of centered complex gaussian variables 
defined on a probability space $(\Omega, {\mathcal A}, p)$ such that $g_{n}=
\overline{g_{-n}}$ and  
$(g_{n}(\omega))_{n>0}$ are independent.
More precisely, we have that for a suitable constant $c$, 
$g_{n}(\omega)=c(h_n(\omega)+{\bf i}l_n(\omega))$,
where $h_n,l_n\in {\mathcal N}(0,1)$ are independent standard real gaussians.
We have that $\mu_{k/2}(H^s)=1$ for every $s<(k-1)/2$ while $\mu_{k/2}(H^{(k-1)/2})=0$.

For any $N\geq 1$, $k\geq 1$ and $R>0$ we introduce the function
\begin{equation}\label{Femme+1}
F_{k/2,N, R}(u)
=\Big(\prod_{j=0}^{k-2} \chi_R (E_{j/2}(\pi_N u)) \Big)
\chi_R (E_{(k-1)/2}(\pi_N u)-\alpha_N) 
e^{-R_{k/2}(\pi_N u)}
\end{equation}
where
$\alpha_N=\sum_{n=1}^N \frac {c}{n}$ for a sutable constant $c$, $\pi_N$ denotes the Dirichlet projector 
on Fourier modes $n$ such that $|n|\leq N$, $\chi_{R}$ is a cut-off function defined as $\chi_R(x)=\chi (x/R)$
with $\chi:\R \rightarrow \R$ a smooth,
compactly supported function such that $\chi(x)=1$ for every $|x|<1$.
For $k=1$ the product term in \eqref{Femme+1} is defined as $1$.
We have the following statement.
%%%%%%%%%%%%%%%%%%%%%%%%%%%%%
\begin{thm}[\cite{tz,TV}]
\label{mainold}
For every $k\in \N$ with $k\geq 1$ there exists a $\mu_{k/2}$ measurable function 
$F_{k/2,R }(u)$ such that $F_{k/2,N, R}(u)$ converges to $F_{k/2,R }(u)$ in $L^q(d\mu_{k/2})$ 
for every $1\leq q<\infty$.
In particular $F_{k/2,R }(u)\in L^q(d\mu_{k/2})$. 
Moreover, if we set $d\rho_{k/2,R}\equiv F_{k/2,R }(u)d\mu_{k/2}$ then we have
$$
\bigcup_{R>0}{\rm supp}(\rho_{k/2,R})={\rm supp}(\mu_{k/2}).
$$
\end{thm}
%%%%%%%%%%%%%%%%%%%%%%%%%%
It is proved in \cite{TV2} that the measures $\rho_{k/2,R}$, given by Theorem 
\ref{mainold}, are invariant along the flow associated with \eqref{bo} provided that $k\geq 6$ and $k$ is even.
There are several difficulties to prove the invariance of the measures $\rho_{k/2,R}$.
The first one is to prove convergence of solutions to the 
finite dimensional approximation of \eqref{bo}, i.e. 
\begin{equation}\label{BONN}
\partial_t u_N + H \partial_x^2 u_N + \pi_N \big ((\pi_Nu_N)\partial_x (\pi_Nu_N)\big )=0,\quad
u(0)=u_0
\end{equation}
to the solution of \eqref{bo}. 
More precisely, if we denote by $\Phi_t^N(u_0)$ the unique global solutions to \eqref{BONN}, then the following estimates are needed in order to prove 
the invariance of $d\rho_{k/2,R}$:
\begin{multline}\label{finitecauchy}\exists s<\sigma<(k-1)/2 \hbox{ s.t. }
\forall\, S>0, \hbox{ } \exists\,  \bar t=\bar t(S)>0 \hbox{ s.t. }
\forall\, \varepsilon>0, 
\\
\Phi_{t}^N(A)\subset \Phi_{t}(A) + B^s(\varepsilon), \hbox{ } \forall\, N>N_0(\varepsilon),
\forall\, t\in (-\bar t, \bar t), \forall A\subset B^\sigma(S), 
\end{multline}
where $B^\sigma(R)$ denotes the ball of radius $R$ and centered at the origin of $H^\sigma$.
%%%%%%
The proof of \eqref{finitecauchy} follows by classical estimates for the Benjamin-Ono equation in the case $k\geq 6$
(see \cite{TV}), and it becomes more and more complicated  as long as $k$ becomes smaller.

A second and more essential source of difficulty, to prove the invariance of $\rho_{k/2,R}$,
is related with the fact that the energies $E_{k/2}$ (see \eqref{strucRmezzi}), that are conserved for the equation \eqref{bo}, 
are no longer conserved for the truncated problems \eqref{BONN}, as long as $k\geq 2$. 
A partial and useful substitute of the lack of invariance of $E_{k/2}$ along the truncated flow
\eqref{BONN} is the following property,  
proved in \cite{TV2} for $k\geq 6$ and $k$ even:
\begin{equation}\label{invDu}\lim_{N\rightarrow \infty}
\sup_{\substack{s\in [0, t]\\A\in {\mathcal B}(H^{(k-1)/2-\epsilon})}}
\Big |\frac d{ds} \int_{\Phi_s^N(A)} F_{k/2,N, R}(u) d\mu_{k/2}(u)\Big |=0\end{equation}
where  ${\mathcal B}(H^{(k-1)/2-\epsilon})$ denote the Borel sets in $H^{(k-1)/2-\epsilon}$.
The first contribution of this paper is the proof of \eqref{invDu} for every $k\geq 2$
(even and odd).
\begin{thm}\label{prcv}
Let $k\geq 2$. 
Then \eqref{invDu} holds for every $\epsilon>0$ sufficiently small, every $t\in\R$ and every $R>0$.
\end{thm}
Notice that we do not consider in Theorem \ref{prcv} the value $k=1$. In fact, as already mentioned  
the energy $E_{1/2}$ (which is indeed the hamiltonian)
is preserved along the truncated flows \eqref{BONN}, and this information
is stronger than \eqref{invDu} for $k=1$.

The novelty in Theorem~\ref{prcv}, compared with \cite{TV2}, is that we allow
$k=2, 4$ and odd $k\geq 3$. In fact, the remaining cases, i.e. $k\geq 6$ and even,
were treated in \cite{TV2}.
The proof of Theorem~\ref{prcv} requires some new ingredients compared with
the argument used in \cite{TV2}. The crucial novelty along the proof of
Theorem \ref{prcv} (in the cases $k\geq 3 $ odd and $k=2,4$) is the use of an extra orthogonality
argument not needed in \cite{TV2}.

In order to explain where the oddness or eveness of $k$ plays a role we recall that, according with Proposition~5.4
in \cite{TV2}, Theorem \ref{prcv}  follows provided that one can show
$$
\lim_{N\rightarrow \infty}\|G_N(u_0)\|_{L^q(d\mu_{k/2})}=0,
\hbox{ where }
G_N(u_0)= \frac d{dt} E_{k/2} (\pi_N \Phi_t^N(u_0))\vert_{t=0}.
$$
Indeed the reduction of the analysis at time $t=0$ is one of the key ideas in
\cite{TV2}. Although the functions
 $G_N(u_0)$ have an explicit expression 
 (see Proposition (3.4) in \cite{TV}), it could be very complicated to deal with them
 due to the intricate algebraic structure of the conservation laws $E_{k/2}$.
 However, once this expression is written and some crucial cancellations
 are done,
 then one is  reduced to prove that the $L^2(dp)$ norm of
 expressions of the following type go to zero as long as $N\rightarrow \infty$:
 \begin{equation}\label{typical}
 \sum_{{\mathcal C}_N} c_{j_1,...,j_n} g_{j_1}...g_{j_n}\end{equation}
 where $c_{j_1,...,j_n}$ are suitable numbers, $g_j$ are the gaussian variables that appear in \eqref{randomized} and
 the dependence on $N$ in \eqref{typical} is hidden in the constraint $
 {\mathcal C}_N$.
The main advantage in the case $k\geq 6$ and $k$ even is that the $L^2(dp)$
estimate of \eqref{typical} can be done via  Minkowski inequality and hence it can be
reduced 
to the analysis of numerical series of the type $
 \sum_{{\mathcal C}_N} |c_{j_1,...,j_n}|$.
 In the case $k=2,4$ and $k\geq 3$ odd the Minkowski inequality
 is useless to estimate \eqref{typical}, and one needs to
 exploit the $L^2(dp)$ orthogonality of multilinear expressions
 $g_{j_1}...g_{j_n}$. Hence 
 we reduce the analysis to numerical expressions of the type
$
 \sum_{{\mathcal C'}_N} |c_{j_1,...,j_n}|^2$,
 where ${\mathcal C'}_N$ is a large subset of ${\mathcal C}_N$. 
 The analysis on the resonant set ${\mathcal C}_N\backslash{\mathcal C'}_N$ is then done in a straightforward way. 

By combining the arguments in \cite{TV2} with Theorem \ref{prcv}, we get the following measure invariance result.
\begin{thm}\label{k>5}
Let $k\geq 4$. Then for every $\epsilon>0$ sufficiently small and every $R>0$,
the measures $\rho_{k/2,R}$ are invariant by the flow 
$\Phi_t$ defined on $H^{(k-1)/2-\epsilon}$.
\end{thm}
The result of Theorem~\ref{k>5} for $k\geq 5$ is a straightforward adaptation of \cite{TV2}. 
Concerning the invariance of the measure $\rho_{2,R}$ 
it is necessary  to prove convergence of solutions for the finite dimensional problems \eqref{BONN} to the 
solution of the original equation \eqref{bo}, at the level of regularity $H^{3/2-\epsilon}$.
Here we perform this analysis following the approach introduced in \cite{KT}. 
This approach has the advantage to be quite flexible and the analysis of the
invariance of  $\rho_{2,R}$ we perform here may be useful to get measure invariance each time the local
well-posedness  on the support of the measure is of quasi-linear nature.

The result of Theorem~\ref{k>5} implies recurrence properties of the Benjamin-Ono on the support of $\mu_{k/2}$.
We refer to the appendix of this paper for a more detailed discussion on this topic. 

The invariance of the measure $\rho_{1/2,R}$
has been proved in \cite{D}. It is worth noticing that the major difficulty in \cite{D}
is the proof of a substitute of \eqref{finitecauchy} in spaces larger than $L^2$ covering the support of $\rho_{1/2,R}$.
On the other hand the advantage of working with $\rho_{1/2,R}$ is that the energy $E_{1/2}$
is preserved along the truncated flows \eqref{BONN} and not just along \eqref{bo}. 

Thanks to the work by Deng \cite{D} and Theorem~\ref{k>5}, it remains to prove the invariance of $\rho_{1,R}$ and  $\rho_{3/2,R}$.
The only remaining point to prove the invariance $\rho_{1,R}$, $\rho_{3/2,R}$ is the proof of \eqref{finitecauchy}. 
In fact, in the case $k=2,3$ the proof of \eqref{finitecauchy}
would require the use of the gauge transformation introduced in \cite{T}
and the more refined Bourgain spaces. This analysis, and even more, is 
essentially contained in the paper \cite{D}. We plan to give the details of this analysis
in a separate paper which will be the last one in the series. 
%%%%%%%%%%%%%%%%%%%%%%%%%%%%%%%%%%%%%%%%%%%%%%%%%%%%%%%%%%%%%%%%%%%%
%%%%%%%%%%%%%%%%%%%%%%%%%%%%%%%%%%%%%%%%%%%%%%%%%%%%%%%%%%%%%%%%%%%%
\section{Some Useful Orthogonality Relations}
The aim of this section is to study orthogonality properties of products of complex centered gaussian variables. For an integer $n\geq 1$, we set
$$
{\mathcal A}_n=\{(j_1, \cdots, j_n)\in (\Z \setminus \{0\})^n \ |\ \sum_{k=1}^n j_k=0 \}
$$
and for every $j\in \Z\setminus \{0\}$  we denote by $g_j(\omega)$
the complex centered Gaussian variable that appears in the $j$-th Fourier 
coefficient in \eqref{randomized}.
Next we give an elementary lemma that will be useful in the sequel.
\begin{lem}\label{tr}
Let $g$ be a complex centered gaussian, i.e. $g=c(h+il)$,  where $c>0$, $h,l\in {\mathcal N}(0,1)$ and $h,l$ are independent.
Then for every $r\neq q$, $r,q\in \N$ we have 
$
\int g^r\overline{g}^q dp=0.
$
\end{lem}
\begin{proof}
By introducing polar coordinates we get
\begin{eqnarray*}
\int g^r\overline{g}^q dp & = &
C\int_{\R^2} (x+iy)^{r}(x-iy)^{q}e^{-\frac{1}{2}(x^2+y^2)}dxdy
\\
& = &
C\int_{0}^{\infty}\int_{0}^{2\pi}r^{p+q+1}e^{i\varphi(r-q)} e^{- \frac{1}{2}r^2}d\varphi dr
=0
\end{eqnarray*}
where at the last step we used that the angular integration vanishes
thanks to the assumption $r\neq q$.
\end{proof}
The next proposition is of importance in order to understand orthogonality of multilinear products of gaussian variables.
\begin{prop}\label{tildeA}
Let
\begin{equation}\label{digraf}
(j_1,...,j_n),(i_1,...,i_n)\in {\mathcal A}_n,
\{j_1,...,j_n\}\neq \{i_1,...,i_n\}\end{equation}
be such that
\begin{equation*}
\int g_{j_1}...g_{j_n}\overline{g_{i_1}
...g_{i_n}} dp\neq 0.
\end{equation*}
Then there exist $1\leq l, m\leq n$, with $l\neq m$ and such that at least one of the following occurs:
either $i_l=-i_m$ or $j_l=-j_m$.
\end{prop}
\begin{proof}
By \eqref{digraf} we get that either
$$\exists\, l\in \{1,...,n\}  \hbox{ s.t. } i_l\notin \{j_1,...,j_n\}$$
or 
$$\exists\, k\in \{1,...,n\} \hbox{ s.t. } j_k\notin \{i_1,...,i_n\}.$$
We assume that we are in the first case (the other one is similar), and
hence let $l$ be fixed as above. Then we introduce 
\begin{eqnarray*}
{\mathcal N}_l & = & \{ k=1,...,n \, | \, i_k=-i_l\},
\\
{\mathcal M}_l & = & \{ k=1,...,n \, | \, i_k=i_l\},
\\
{\mathcal L}_l & = & \{ k=1,...,n \, | \, j_k=-i_l\}.
\end{eqnarray*}
Notice that ${\mathcal M}_l\neq \emptyset$ since it contains at least the element $l$.
Our aim
is to prove that ${\mathcal N}_l\neq \emptyset$.
Assume by the absurd that 
${\mathcal N}_l=\emptyset$, then by independence and by Lemma~\ref{tr} we get
$$\int g_{j_1}...g_{j_n}\overline{g_{i_1}
...g_{i_n}} dp=\int 
{\overline g_{i_l}}^{|{\mathcal M}_l| + 
|{\mathcal L}_l|} dp \int \big( \Pi_{k\notin {\mathcal M}_l} \overline{g_{i_k}}\big) 
\big(\Pi_{h\notin {\mathcal L}_l} g_{j_h}\big)
dp=0$$ where
%${\mathcal M}_l^c=\{1,...,n\} \setminus {\mathcal M}_l$,
%${\mathcal L}_l^c=\{1,...,n\} \setminus {\mathcal L}_l$, 
$|{\mathcal M}_l|$ and $|{\mathcal L}_l|$ denote the cardinality of ${\mathcal M}_l$ and ${\mathcal L}_l$.
We get a contradiction,
therefore ${\mathcal N}_l\neq \emptyset$.
\end{proof}
Motivated by Proposition~\ref{tildeA} we introduce the following sets:
$$\tilde {\mathcal A}_n=\{(j_1,..., j_n)\in {\mathcal A}_n  \,|\,  j_k\neq -j_l 
\hbox{ } \forall k,l\}
\hbox{ and }
\tilde {\mathcal A}_n^c={\mathcal A}_n\setminus \tilde {\mathcal A}_n.
$$
In particular we get 
\begin{equation}\label{unionA5}
{\mathcal A}_n= \tilde {\mathcal A}_n \cup \tilde {\mathcal A}_n^c
\hbox{ and the union is disjoint. }
\end{equation}
\begin{remark}\label{bersstr}
We shall need to consider along the paper the following sets,
where $N\in \N$:
$$
\{(j_1,j_2, j_3, j_4)\in \tilde {\mathcal A}_4^{c} \,\vert \,
|j_1+j_2|>N\}.
$$
It is easy to check that this set can be characterized as follows:
$$\big \{(k, h, -k,-h), (k, h, -h, -k)| h, k\in \Z\setminus \{0\}, |h+k|>N\big \}.
$$
\end{remark}
\begin{cor} Assume that 
\begin{equation*}
(j_1,j_2,j_3),(i_1, i_2,i_3)\in {\mathcal A}_3,
\{j_1,j_2,j_3\}\neq \{i_1,i_2,i_3\}\end{equation*}
then
$
\int g_{j_1}g_{j_2}g_{j_3}\overline{g_{i_1}
g_{i_2}g_{i_3}} dp= 0$.
\end{cor}
\begin{proof}
It follows by Proposition \ref{tildeA},
in conjunction with the fact that $\tilde {\mathcal A}_3={\mathcal A}_3$.
To prove the last identity assume by the absurd that there exists $(i_1, i_2, i_3)\in \tilde {\mathcal A}_3$, such that 
$i_l=-i_m$ for some $1\leq l, m\leq 3, l\neq m$. Then this implies, by the condition
$i_1+i_2+i_3=0$, that necessarily $i_k=0$ for $k\in \{1, 2, 3\}\setminus \{l, m\}$. We get an
absurd since by definition of 
${\mathcal A}_3$ we have $i_1, i_2, i_3\neq 0$.
\end{proof}
%%%%%%%%%%%%%%%%%%%%%%%%%%
\begin{cor}\label{5,5}
Assume that 
$$
(j_1,..., j_n), (i_1,...,i_n)\in \tilde {\mathcal A}_n,
\{j_1,...,j_n\}\neq \{i_1,..., i_n\}$$
then 
$
\int g_{j_1}...g_{j_n}
\overline{g_{i_1}...g_{i_n}
} dp=0.
$
\end{cor}
\begin{proof}
It follows by Proposition \ref{tildeA}.
\end{proof}
Next we introduce for every $j\in \Z\setminus \{0\},n\in \N$
\begin{equation}\label{Acj}\tilde {\mathcal A}^{c,j}_n=\{(j_1,...,j_n)\in \tilde {\mathcal A}_n^c| j=j_l=-j_m \hbox{ for some } 1\leq l\neq m\leq n\}.\end{equation}
Hence we get
\begin{equation}\label{disjAc5}
\tilde {\mathcal A}^c_n= \bigcup_{j>0} \tilde {\mathcal A}^{c,j}_n.\end{equation}
Notice that in \eqref{disjAc5} the union is for positive $j$ since 
by definition we have 
\begin{equation}\label{dommat}
\tilde {\mathcal A}^{c,j}_n=\tilde {\mathcal A}^{c,|j|}_n \hbox{ } \forall \, j\in \Z\setminus\{0\}.\end{equation}
%%%%%%%%%%%%%%%%%%%%%%%%%
\begin{remark}\label{rem5}
In general the union in \eqref{disjAc5}
is not disjoint. However it is disjoint for $n=5$.
In fact assume that $(j_1, j_2, j_3, j_4,j_5)\in \tilde {\mathcal A}^{c,i}_5\cap \tilde 
{\mathcal A}^{c,j}_5$ for $i\neq j$, $i,j>0$. Then up to permutation we can assume $j_1=-j_2=j$, 
$j_3=-j_4=i$ and hence by the condition $j_1+j_2+j_3+j_4+j_5=0$ we get
$j_5=0$, which is in contradiction with the hypothesis $(j_1, j_2, j_3, j_4,j_5)\in {\mathcal A}_5$.
\end{remark}
%%%%%%%%%%%%%%%%%%%%%%%%%%%%%%%%%
\begin{prop}\label{forp=2}
Let $i, j>0$ be fixed and assume \begin{align}\nonumber&(j_1, j_2, j_3, j_4,j_5)\in \tilde {\mathcal A}^{c,j}_5,
(i_1,i_2, i_3, i_4,i_5)\in \tilde {\mathcal A}^{c,i}_5\\\nonumber
&\big \{\{j_1, j_2, j_3, j_4,j_5\}\setminus \{j, -j\}\big\}\neq \big \{\{i_1, i_2, i_3, i_4,i_5\}\setminus \{i, -i\}\big \}\end{align}
then
$
\int g_{j_1}g_{j_2}g_{j_3}g_{j_4}g_{j_5}
\overline{g_{i_1}
g_{i_2}g_{i_3}
g_{i_4}g_{i_5}
} dp=0$.
\end{prop}

\begin{proof}
For simplicity we can assume
$j=j_1=-j_2$ and $i=i_1=-i_2$. Since by assumption 
$$\big \{\{j_1, j_2, j_3, j_4,j_5\}\setminus \{j, -j\}\big\}\neq \big \{\{i_1, i_2, i_3, i_4,i_5\}\setminus \{i, -i\}\big \}$$
we deduce as before that either
$$\exists\, l\in \{3,4,5\} \hbox{ s.t. } i_l\notin \{j_3,j_4,j_5\}$$
or
$$\exists\, l\in \{3,4,5\} \hbox{ s.t. } j_l\notin \{i_3,i_4,i_5\}.$$
Again we only consider the first possibility, the analysis of the second being identical.  
We introduce the sets
$${\mathcal N}_l=\{ k=3,4,5| i_k=-i_l\},
{\mathcal M}_l=\{ k=3,4,5| i_k=i_l\},
{\mathcal L}_l=\{ k=3,4,5| j_k=-i_l\}.$$
Notice that ${\mathcal N}_l=\emptyset$ (otherwise by the constrained $
i_1+i_2+i_3+i_4+i_5=0$ and by recalling that $i=i_1=-i_2$, we would get
$i_m=0$ where $m\in \{1,2,3,4,5\}\setminus \{1,2, l, k\}$, which is absurd 
since $(i_1,i_2, i_3, i_4, i_5)\in {\mathcal A}_5$).
Hence we get
$$
\int g_{j_1}g_{j_2}g_{j_3}g_{j_4}g_{j_5}
\overline{g_{i_1}g_{i_2}g_{i_3}g_{i_4}
g_{i_5}} dp$$$$=\int |g_j|^2 |g_i|^2 
\overline{g_{i_l}}^{|{\mathcal M}_l| + 
|{\mathcal L}_l|} 
\big( \Pi_{k\notin {\mathcal M}_l} 
\overline{g_{i_k}}\big) 
\big(\Pi_{h\notin {\mathcal L}_l} g_{j_h}\big) 
dp
$$
where 
$|{\mathcal M}_l|, |{\mathcal L}_l|$ denote the cardinality of ${\mathcal M}_l, {\mathcal L}_l$.
By combining independence with Lemma \ref{tr}, and by noticing that
by definition $|{\mathcal M}_l|\geq 1$, we can continue the identity above as follows:
$$...= \int |g_j|^2 |g_i|^2 g_{i_l}^{|{\mathcal M}_l| + 
|{\mathcal L}_l|} dp \int 
\big( \Pi_{k\notin {\mathcal M}_l} 
\overline{g_{i_k}}\big) 
\big(\Pi_{h\notin {\mathcal L}_l} g_{j_h}\big) 
dp=0 \hbox{ if }|i_l|=|i|=|j|
$$
%%%%%
$$...= \int |g_i|^2 g_{i_l}^{|{\mathcal M}_l| + 
|{\mathcal L}_l|} dp \int  |g_j|^2
\big( \Pi_{k\notin {\mathcal M}_l} 
\overline{g_{i_k}}\big) 
\big(\Pi_{h\notin {\mathcal L}_l} g_{j_h}\big) 
dp=0 \hbox{ if }|i_l|=|i|\neq |j|
$$
%%%%%%%%
$$
...= \int |g_j|^2 g_{i_l}^{|{\mathcal M}_l| + 
|{\mathcal L}_l|} dp \int  |g_i|^2
\big( \Pi_{k\notin {\mathcal M}_l} 
\overline{g_{i_k}}\big) 
\big(\Pi_{h\notin {\mathcal L}_l} g_{j_h}\big) 
dp=0 \hbox{ if }|i_l|=|j|\neq|i|
$$
%%%%%%%%%
$$...= \int  g_{i_l}^{|{\mathcal M}_l| + 
|{\mathcal L}_l|} dp \int  |g_j|^2 |g_i|^2
\big( \Pi_{k\notin {\mathcal M}_l} 
\overline{g_{i_k}}\big) 
\big(\Pi_{h\notin {\mathcal L}_l} g_{j_h}\big) 
dp=0 \hbox{ if }|i_l|\neq |j|, |i|$$
and we thus conclude the proof.
\end{proof}
%%%%%%%%%%%%%%%%%
As a consequence we get the following useful corollary. 
\begin{cor}\label{orthtzv}
Let  $j\in \N\setminus\{0\}$ be fixed and 
$$(j_1, j_2,j_3,j_4, j_5), (i_1,i_2,i_3,i_4, i_5)\in \tilde {\mathcal A}_5^{c,j}, \hbox{ } \{j_1,j_2,j_3,j_4, j_5\}\neq \{i_1,i_2,i_3,i_4, i_5\}$$
then
$
\int g_{j_1}g_{j_2}g_{j_3}g_{j_4}g_{j_5}
\overline{g_{i_1}
g_{i_2}g_{i_3}
g_{i_4}g_{i_5}
} dp=0$.
\end{cor}
\begin{proof}
Apply Proposition \ref{forp=2} with $i=j$.
\end{proof}
We conclude this section with some notations that will be useful in the sequel.
We need to introduce some subsets 
of $\tilde {\mathcal A}^{c,j}_n$ (see \eqref{Acj}):
given any couple $1\leq l<m\leq n$ and any fixed $j\in \Z\setminus \{0\}$
we define
$$\tilde {\mathcal A}^{c,j, (l,m)}_{n}= \{(j_1,...,j_n)\in \tilde {\mathcal A}_n^{c,j}|
j=j_l=-j_m\}.$$
Notice that, since we have fixed an ordering on $l,m$, it is no longer true that 
$\tilde {\mathcal A}^{c,j, (l,m)}_{n}=\tilde {\mathcal A}^{c,-j, (l,m)}_{n}$
(compare with \eqref{dommat}). Actually $\tilde {\mathcal A}^{c,j, (l,m)}_{n}$
and $\tilde {\mathcal A}^{c,-j, (l,m)}_{n}$ are disjoint.

Of course we have that 
$$\tilde {\mathcal A}^{c,j}_n= 
\bigcup_{0<l<m\leq n} (\tilde {\mathcal A}^{c,j, (l,m)}_{n} \cup 
\tilde {\mathcal A}^{c,-j, (l,m)}_{n}).$$
Notice that the union above is not disjoint.
For instance we have
$$(j, j, -j, k, -j-k)\in \tilde A^{c,j, (1,3)}_{6} \cap \tilde A^{c,j, (2,3)}_{6}.$$
In order to overcome this difficulty we introduce the usual ordering on the couples $(l, m)$ (i.e. $(l_1, m_1)<(l_2, m_2)$
if $l_1<l_2$ or $l_1=l_2$ and $m_1<m_2$).
Once this ordering is introduce we define
\begin{equation*}{\mathcal B}^{c,j, (1,2)}_{n}=\bigcup_\pm \tilde {\mathcal A}^{c,\pm j, (1,2)}_{n}\end{equation*}
and by induction
\begin{equation*}
{\mathcal B}^{c,j, (l_0,m_0)}_{n}=\bigcup_\pm \tilde {\mathcal A}^{c,\pm j, (l_0,m_0)}_{n}
\setminus \bigcup_{(l,m)<(l_0, m_0)} {\mathcal B}^{c,j, (l,m)}_{n}.
\end{equation*}
Hence we get
\begin{equation}\label{dijfinfat}
\tilde {\mathcal A}^{c,j}_n= \bigcup_{0<l<m\leq n} 
{\mathcal B}^{c,j, (l,m)}_{n},
\end{equation}
where the union is disjoint.
%%%%%%%%%%%%%%%%%%%%%%%%%%%%%%%%%%%%%%%%%%%%%%%%%%%%%%%%%%%%%%
\section{Some calculus inequalities}
As a matter of convention, in the sequel, we assume that a summation on the empty set equals to zero.
The following lemma has been crucial in \cite{TV} and \cite{TV2}.
\begin{lem}\label{algebrTV} The following estimate occurs:
$$\sum_{\substack{|j+l|>N\\0<|j|, |l|\leq N}} 
\frac 1{|j||l|^2}=O\big (\frac{\ln N}{N}\big).$$
\end{lem}
The next lemma will be useful in the sequel. 
\begin{lem}\label{algebrTV2}
For any $m\geq 2$, the following estimate occurs:
\begin{equation}\label{priTVtotr}\sum_{\substack{|\sum_{k=1}^m j_k|>N\\0<|j_1|,...,|j_m|\leq N}} 
\frac 1{|j_1| \Pi_{k=2}^m |j_k|^2}=O\big (\frac{\ln N}{N}\big).
\end{equation}
\end{lem}
\begin{proof}

Notice that $$|\sum_{k=1}^m j_k|>N \Rightarrow \sum_{k=1}^m |j_k|>N$$
hence it is sufficient to prove \eqref{priTVtotr} under the extra condition
$j_1,..., j_m>0$.
Next we use induction on $m$ (notice that the case $m=2$ is the content of Lemma
\ref{algebrTV}). Hence we show that if we assume
\begin{equation}\label{sumPritO}
\sum_{\substack{\sum_{k=1}^{m-1} j_k>N\\0<j_1,...,j_{m-1}\leq N}} 
\frac 1{j_1 \Pi_{k=2}^{m-1} j_k^2}=O\big (\frac{\ln N}{N}\big)
\end{equation}
then \eqref{priTVtotr} (where we remove the absolute values $|.|$) is true.
Notice that the sum in \eqref{priTVtotr} can be splitted as follows:
$$\sum_{\substack{\sum_{k=1}^{m-1} j_k>N\\0<j_1,..., j_m\leq N}} 
\frac 1{j_1 \Pi_{k=2}^m j_k^2}+ \sum_{\substack{
0<\sum_{k=1}^{m-1} j_k\leq N\\
N- \sum_{k=1}^{m-1} j_k<j_m\leq N \\
0<j_1,..., j_{m-1}\leq N}} 
\frac 1{j_1 \Pi_{k=2}^m j_k^2}
=I_N+II_N.$$
By \eqref{sumPritO} we get immediately
$I_N=O \big(\frac{\ln N}{N}\big)$, hence the proof will be complete
if we show 
\begin{equation}\label{merfra}II_N=O\Big( \frac{\ln N}{N}\Big).
\end{equation}
We have that there exists $C$ such that for every $1\leq a\leq N$,
\begin{equation}\label{integral}
\sum_{k=a}^{N}\frac{1}{k^2}\leq \frac{C(N-a+1)}{aN}\,.
\end{equation}
The bound \eqref{integral} can be obtained by evaluating the sum by an integral.
Therefore, we obtain
$$II_N \leq \frac CN  \sum_{\substack{
0<\sum_{k=1}^{m-1} j_k\leq N\\
0<j_1,..., j_{m-1}\leq N}} 
\frac 1{j_1 \Pi_{k=2}^{m-1}j_k^2} 
\Big (\frac{\sum_{k=1}^{m-1} j_k}{N- \sum_{k=1}^{m-1} j_k+1}\Big).
$$
Hence the proof of \eqref{merfra} (and as a consequence of \eqref{priTVtotr}) will be completed, provided that:
\begin{equation}\label{hnMtz}\sum_{\substack{
0<\sum_{k=1}^{m-1} j_k\leq N\\
0<j_1,..., j_{m-1}\leq N}} 
\frac 1{\Pi_{k=2}^{m-1}j_k^2} 
\Big (\frac{1}{N- \sum_{k=1}^{m-1} j_k+1}\Big)=O(\ln N)
\end{equation}
and
\begin{equation}\label{bermonber}
\sum_{\substack{0<\sum_{k=1}^{m-1} j_k\leq 
N\\0<j_1,..., j_{m-1}\leq N}} \frac 1{j_1 j_l \Pi_{\substack{k=2\\k\neq l}}^{m-1}j_k^2} 
\Big ( \frac{1}{N- \sum_{k=1}^{m-1} j_k+1}\Big)=O(\ln N),
\end{equation}
for every $l=2,...,m-1$.
Notice that the l.h.s. in 
\eqref{hnMtz} can be estimated as follows  
$$...\leq \sum_{\substack{%0<\sum_{k=1}^{m-1} j_k\leq  N\\
0<j_2,..., j_{m-1}\leq N}} 
\frac 1{\Pi_{k=2}^{m-1}j_k^2} \Big(\sum_{0<j_1\leq N - \sum_{k=2}^{m-1}j_k}
\frac{1}{N- j_1- \sum_{k=2}^{m-1} j_k +1}\Big )=O(\ln N)$$
and hence we get \eqref{hnMtz} .
Concerning \eqref{bermonber} we can assume by symmetry that $l=2$, and in this case
the r.h.s. in \eqref{bermonber} can be written as follows:
$$...=\sum_{J=1}^N \frac 1{(N- J+1)}\sum_{\substack{\sum_{k=1}^{m-1} j_k=J\\
0< j_1, j_2,...,j_{m-1}\leq N
}}  \frac 1{j_1j_2 \Pi_{k=3}^{m-1} j_k^2}$$
(for $m=3$ the product $\Pi_{k=3}^{m-1}$ is assumed $1$).
As a consequence  \eqref{bermonber} follows provided that
\begin{equation}\label{leofitzv}
\sup_{N\geq 1} \Big (\sup_{1\leq J\leq N} \sum_{\substack{\sum_{k=1}^{m-1} j_k=J\\
0< j_1, j_2,...,j_{m-1}\leq N
}}  \frac 1{j_1j_2 \Pi_{k=3}^{m-1} j_k^2}\Big)<\infty.
\end{equation}
For fixed $N$ and $J$, we can write
\begin{eqnarray*}
\sum_{\substack{\sum_{k=1}^{m-1} j_k=J\\
0< j_1, j_2,...,j_{m-1}\leq N
}}  \frac 1{j_1j_2 \Pi_{k=3}^{m-1} j_k^2}
& \leq &
2\sum_{\substack{\sum_{k=1}^{m-1} j_k=J\\
0< j_1, j_2,...,j_{m-1}\leq N
}}  
\Big(
\frac{1}{j_1^2}+\frac{1}{j_2^2}
\Big)
\frac 1{ \Pi_{k=3}^{m-1} j_k^2}
\\
& \leq &
4\Big(\sum_{j>0}\frac{1}{j^2}\Big)^{m-2}
\\
& \leq & C.
\end{eqnarray*}
This completes the proof of Lemma~\ref{algebrTV2}.
%The estimate above is equivalent to
%\begin{equation}\label{kmer}
%\sup_{N\geq 1} \Big (\sup_{1\leq J\leq N} \sum_{K=1}^J\sum_{\substack{j_1+j_2=K\\
%\sum_{k=3}^{m-1} j_k=J-K\\
%0\leq j_1, j_2,...,j_{m-1}\leq N
%}}  \frac 1{j_1j_2 \Pi_{k=3}^{m-1} j_k^2}\Big)<\infty
%\end{equation}
%which in turn follows by 
%$$ \sum_{K=1}^J \Big(\sum_{\substack{
%j_1+j_2=K\\
%\sum_{k=3}^{m-1} j_k=J-K\\
%0\leq j_2,...,j_{m-1}\leq N\\
%1<j_1\leq K
%}}  \frac 1{j_1(K-j_1) \Pi_{k=3}^{m-1} j_k^2}\Big)$$
%$$\leq  C \sum_{K=1}^J \Big( \sum_{\substack{
%j_1+j_2=K\\
%\sum_{k=3}^{m-1} j_k=J-K\\
%0\leq j_2,...,j_{m-1}\leq N\\
%1<j_1\leq K
%}}  \frac 1{j_1(K-j_1)(J-K)}\Big)$$
%where in the last estimate we have used the fact that
%$$\sum_{k=3}^{m-1} j_k=J-K\Rightarrow \sup_{k=3,...,m-1} \{j_k\}\geq \frac{J-K}{m-4}$$
%We can continue the last estimate as follows:
%$$...\leq C\sum_{K=1}^J \frac{\ln K}{K(J-K)} = O\Big( \frac{\ln^2J }{J}\Big)$$
%and hence we get \eqref{kmer}, which in turn implies \eqref{leofitzv}, and we conclude the proof.
\end{proof}
\begin{lem}\label{serienew}
The following estimate occurs:
$$\sum_{0<|j|\leq N}
\sqrt { \sum_{\substack{|j+l|>N\\ 0<|l|\leq N}} \frac 1{j^2 l^2} }=O\big(\frac 1{\sqrt N} \big).$$
\end{lem}

\begin{proof}
As in the proof of Lemma~\ref{algebrTV2}
it is sufficient to prove the estimate 
under the extra condition $0<j,l\leq N$.
Using \eqref{integral}, we obtain that for every $j>0$ ($j\leq N$) fixed one has the bound:
$$
\sum_{l=N-j+1}^{N}
%\sum_{\substack{j+l>N\\ 0<l<N}} \frac{1}{j^2l^2}
%\leq C \frac 1{j^2} \big( \frac 1{N+1-j} - \frac 1N\big )
\leq  \frac{Cj}{j^2(N+1-j)N}.$$
Hence we get
$$\sum_{0<j\leq N}
\sqrt { \sum_{\substack{j+l>N\\ 0<l\leq N}} \frac 1{j^2 l^2} }
\leq C \frac 1{\sqrt N} \sum_{0<j\leq N} \frac 1{\sqrt {j (N+1-j)}}.$$
The proof of lemma follows by the following chain of inequalities:
\begin{eqnarray*}
\sum_{0<j\leq N} \frac 1{\sqrt {j (N+1-j)}}
& = & \sum_{0<j\leq N/2}\dots
% \frac 1{\sqrt {j (N+1-j)}}
+\sum_{N/2<j\leq N} \dots
%\frac 1{\sqrt {j (N+1-j)}}
\\
& \leq & \sum_{0<j\leq N/2} \frac C{\sqrt {jN}} + \sum_{N/2<j\leq N} \frac C{\sqrt {N (N+1-j)}}
\\
& \leq &
 \frac C{\sqrt N} \sum_{0<j\leq N} \frac 1{\sqrt {j}}
 \\
& \leq & C.
\end{eqnarray*}
This completes the proof of Lemma~\ref{serienew}.
\end{proof}
%%%%%%%%%%%
\begin{lem}\label{sersaut}
The following estimate occurs:
\begin{equation*}\sum_{\substack{|j+l|>N\\ 0<|j|,|l|\leq N}} \frac{|j|}{l^2}=O(N\ln N).
\end{equation*}
\end{lem}
\begin{proof} We can assume, arguing as in Lemma~\ref{algebrTV2}, that $j,l>0$.
Hence by elementary computations we get:
$$\sum_{0<l\leq N} \frac{1}{l^2} \sum_{N-l <j\leq N} j
%= \frac 12 \sum_{0<l\leq N} \frac{1}{l^2} \sum_{N-l <j\leq N} j
%$$$$
=\frac 12 \sum_{0<l\leq N} \frac{1}{l^2} [N(N+1) - (N-l)(N-l+1)]$$
$$
= \frac 12 \sum_{0<l\leq N} \frac{1}{l} (2N+1-l)=O(N\ln N) .$$
This completes the proof of Lemma~\ref{sersaut}.
\end{proof}
%%%%%%%%%%%%%%%%%%%%%%%%%%%%%%%%%%%%%%%%%%%%%%%%
%%%%%%%%%%%%%%%%%%%%%%%%%%%%%%%%%%%%%%%%%%%%%%%%
\section{Proof of Theorem \ref{prcv} for $k=2$}\label{k=2}
First we recall the explicit expression of the energy
\begin{equation}\label{E1jk}E_1(u)  =  \|u\|_{\dot{H}^1}^2 + \frac 34  \int u^2 H(u_x) dx 
+ \frac 18 \int u^4 dx.\end{equation}
Next we introduce the sequence of functions $G_N^{k/2}$, $k=0,1,2$, defined as follows:
$$G_N^1: {\rm supp}(\mu_1)\ni u\mapsto \frac{d}{dt} E_{1}\Big (\pi_N \Phi_t^N(u)\Big )_{t=0}$$
$$G_N^{1/2}: {\rm supp}(\mu_1)\ni u\mapsto \frac{d}{dt} E_{1/2}\Big (\pi_N \Phi_t^N(u)\Big )_{t=0}$$
$$G_N^0: {\rm supp}(\mu_1)\ni u\mapsto \frac{d}{dt} \Big({\|\pi_N \Phi_t^N(u)\|_{L^2}^2}\Big)_{t=0}$$
where $\pi_N$ is the projector on the Fourier modes $n$ such that $|n|\leq N$.
According to Proposition~5.4 in \cite{TV2}, Theorem~\ref{prcv} for $k=2$ 
follows from the following proposition.
\begin{prop}\label{necv}
We have
\begin{equation*}
\lim_{N\to \infty} \sum_{k=0}^2 \|G_N^{k/2}(u)\|_{L^2(d\mu_1(u))}=0.
\end{equation*}
\end{prop}
\begin{proof}[Proof of Proposition~\ref{necv}]
The property 
$\lim_{N\to \infty} \sum_{k=0}^1 \|G_N^{k/2}(u)\|_{L^2(d\mu_1(u))}=0$
follows from the fact that $E_{1/2}(\pi_N \Phi_t^N(u))\equiv {\rm const}$ and $\|\pi_N \Phi_t^N(u)\|_{L^2}\equiv {\rm const}$.
Hence we shall focus on the proof of $\|G_N^{1}(u)\|_{L^2(d\mu_1(u))}
\rightarrow 0$ as $N\rightarrow \infty$.
In the sequel we shall use the notation $\pi_{>N}\equiv Id - \pi_N$.
We have the following explicit expression of $G^1_N$.
\begin{lem}\label{ani}
For every $N\in \N$ and for every $u\in {\rm supp}(\mu_1)$ we have
 \begin{equation}\label{nicexp725}G_N^1(u)= -\frac 34 \int (\pi_N u)^2
 \partial_x (\pi_N u)  \pi_{>N}\big( (\pi_N u)^2\big)dx.
 \end{equation}
 \end{lem} 
\begin{proof}
First we recall that, by Proposition 3.4 in \cite{TV2} and \eqref{E1jk}, we have:
\begin{equation}\label{hebSR}\frac{d}{dt} E_{1}\Big (\pi_N \Phi_t^N(u)\Big )_{t=0}=
\frac 3 2 \int u_N \pi_{> N} (u_N \partial_x u_N) H (\partial_x u_N) dx 
\end{equation}
$$+ \frac 34 \int (u_N)^2 H \pi_{>N} \partial_x (u_N  \partial_xu_N) dx 
+\frac 12 \int \pi_{>N} (u_N  \partial_x u_N) u_N^3 dx
$$$$=I_N(u)+II_N(u)+III_N(u)$$
where $u_N=\pi_N u$.
Notice that if we denote by $u^+_N$ (resp. $u^-$) the projection of $u_N$ on the positive
(resp. negative)
Fourier modes then we have $\pi_{>N} (u_N^+ u_N^-)=0$ and hence we can deduce (see \cite{TV}, \cite{TV2}
for more details) the following identity:
$$I_N(u)=\frac 3 2 \int \pi_{> N} (u_N^+ H (\partial_x u_N^+)) \pi_{> N} (u_N^- \partial_x u_N^-) dx 
$$$$+ \frac 3 2 \int \pi_{> N} (H (\partial_x u_N^-) u_N^-) \pi_{> N} (u_N^+ \partial_x u_N^+)  dx $$
and hence by definition of $H$
$$ ...=-i \frac 3 2 \int \pi_{> N} (u_N^+ \partial_x u_N^+) \pi_{> N} (u_N^- \partial_x u_N^-) dx 
$$$$+ i \frac 3 2 \int \pi_{> N} (\partial_x u_N^- u_N^-) \pi_{> N} (u_N^+ \partial_x u_N^+)  dx=0.$$
Similarly we have
$$II_N(u)= -\frac 34 \int \partial_x (u_N^+)^2 H \pi_{>N} (u_N^-  \partial_xu_N^-) dx 
-\frac 34 \int \partial_x (u_N^-)^2 H \pi_{>N} (u_N^+  \partial_xu_N^+) dx 
$$
$$=- i\frac 38 \int \partial_x \pi_{>N}(u_N^+)^2 \partial_x \pi_{>N} (u_N^-)^2   dx  
+ i\frac 38  \int \partial_x \pi_{>N}(u_N^-)^2 \partial_x  \pi_{>N} (u_N^+)^2 dx =0.$$
Hence we get by \eqref{hebSR}
$$\frac{d}{dt} E_{1}\Big (\pi_N \Phi_t^N(u)\Big )_{t=0}=\frac 12 \int \pi_{>N} (u_N  \partial_x u_N) u_N^3 dx.$$
The conclusion follows by integration by parts.
This completes the proof of Lemma~\ref{ani}.
\end{proof}
%%%%%%%%%%%%%%%%%%%%%%
Let us come back to the proof of Proposition~\ref{necv}.
In order to prove $\|G_N^1(u)\|_{L^2(d\mu_1(u))}$ tends to zero as $N\rightarrow \infty$,
we plug in the r.h.s. of \eqref{nicexp725} the random vector
\eqref{randomized}  (where we fix $k=2$)  
and we are reduced to prove (see \cite{TV} and \cite{TV2} for similar reductions
in the case of higher order conservation laws)
\begin{equation}\label{A567890}
\Big \| \sum_{\substack
{0<|j_1|, 
|j_2|, |j_3|, |j_4|, |j_5|\leq N\\
|j_4+j_5|>N\\(j_1,j_2,j_3,j_4,j_5)\in {\mathcal A}_5}}
\frac{j_3 g_{j_1}g_{j_2}g_{j_3}g_{j_4}g_{j_5}}{|j_1||j_2||j_3||j_4||j_5|}
 \Big\|_{L^2(dp)}
\rightarrow 0 \hbox{ as }N\rightarrow \infty.\end{equation}
Notice that by combining \eqref{unionA5} with the Minkowski inequality,
it is sufficient to prove that:
$$\|F_N\|_{L^2(dp)}, \|F_N^{c}\|_{L^2(dp)}
\rightarrow 0$$
where the functions $F_N$ and $F_N^c$ are defined as in \eqref{A567890} with the condition $\vec j\in {\mathcal A}_5$
replaced respectively by $\vec j\in \tilde {\mathcal A}_5$ and
$\vec j\in \tilde {\mathcal A}_5^c$ (where we used the notation $\vec j=(j_1, j_2, j_3, j_4, j_5)$).
In order to estimate $\|F_N\|_{L^2(dp)}$ we use Corollary \ref{5,5}
and by orthogonality we get:
$$\|F_N\|_{L^2(dp)}^2\leq C
\sum_{\substack{0<|j_1|, |j_2|, |j_4|, |j_5|\leq N\\|j_4+j_5|>N}} \frac{1}{|j_1|^2 |j_2|^2 |j_4|^2 |j_5|^2}=O\Big( \frac 1N\Big)$$
where the last step we used the fact that $|j_4+j_5|>N$ implies
$\max\{|j_4|, |j_5|\}>N/2$.\\
Next we estimate $\|F_N^{c}\|_{L^2(dp)}$,
i.e.
\begin{equation*}
\Big \| \sum_{\substack
{0<|j_1|, 
|j_2|, |j_3|, |j_4|, |j_5|\leq N\\
|j_4+j_5|>N\\(j_1,j_2,j_3,j_4,j_5)\in \tilde {\mathcal A}_5^{c}}}
\frac{j_3g_{j_1}g_{j_2}g_{j_3}g_{j_4}g_{j_5}}{|j_1||j_2||j_3||j_4||j_5|}
\Big\|_{L^2(dp)}\rightarrow 0 \hbox{ as } N\rightarrow \infty.\end{equation*}
Notice that by \eqref{disjAc5} (where the union is disjoint for $n=5$ by Remark~\ref{rem5})
and by Minkowski inequality, it is sufficient to prove that
\begin{equation*}
%\label{A5678909998}
\sum_{j=1}^N \Big \| \sum_{\substack
{0<|j_1|, 
|j_2|, |j_3|, |j_4|, |j_5|\leq N\\
|j_4+j_5|>N\\(j_1,j_2,j_3,j_4,j_5)\in \tilde {\mathcal A}_5^{c,j}}}
\frac{j_3g_{j_1}g_{j_2}g_{j_3}g_{j_4}g_{j_5} }{|j_1||j_2||j_3||j_4||j_5|}
\Big\|_{L^2(dp)}
\rightarrow 0 \hbox{ as } N\rightarrow \infty\end{equation*}
that in turn due to  \eqref{dijfinfat}
follows by 
\begin{equation}\label{A5678909998tz}
\sum_{j=1}^N \Big \| \sum_{\substack
{0<|j_1|, 
|j_2|, |j_3|, |j_4|, |j_5|\leq N\\
|j_4+j_5|>N\\(j_1,j_2,j_3,j_4,j_5)\in {\mathcal B}_5^{c,j, (l,m)}}}
\frac{j_3g_{j_1}g_{j_2}g_{j_3}g_{j_4}g_{j_5} }{|j_1||j_2||j_3||j_4||j_5|}
 \Big\|_{L^2(dp)}
\rightarrow 0 \end{equation}
with $(l, m)=(1,2), (1,3), (1,4) ,(1,5), (2,3) ,(2,4) ,(2,5), (3,4), (3,5)$.
Indeed we have excluded the possibility $(l,m)=(4,5)$ since
for every $(j_1, j_2, j_3, j_4, j_5)\in {\mathcal B}_5^{c,j, (4,5)}$
we get $|j_4+j_5|=|j-j|=0$ which is in contradiction with
the constrained $|j_4+j_5|>N$ that appears in 
\eqref{A5678909998tz}.
Notice that we can also exclude $(l,m)=(1,2)$ since for any $(j_1, j_2,j_3, j_4, j_5)\in 
{\mathcal B}_5^{c,j, (1,2)}$ with $|j_4+j_5|>N$, we have 
$N<|j_4+j_5|=|j_1+j_2+j_3|=|j-j+j_3|=|j_3|$, which contradicts the condition $|j_3|\leq N$
that appears in \eqref{A5678909998tz}.
By similar argument we can exclude $(l,m)=(1,3), (2,3)$.
Next we prove \eqref{A5678909998tz} by considering separately all the possible
remaining values for the couple $(l,m)$.
\begin{itemize}
%\item $(l,m)=(1,2)$, then by using the orthogonality stated in Corollary \ref{orthtzv}, the estimate \eqref{A5678909998tz} follows by \begin{equation*}
%\sum_{j_1=1}^N  \Big( \sum_{\substack {0<  |j_4|, |j_5|\leq N\\
%|j_4+j_5|>N}} \frac{1}{|j_1|^4|j_4|^2|j_5|^2} \Big)^{\frac 12}\end{equation*}$$\leq \sum_{j_1=1}^N \frac 1{j_1^2} \Big (\sum_{\substack {0< |j_4|, |j_5|\leq N\\ |j_4+j_5|>N}} \frac{1}{|j_4|^2|j_5|^2} \Big )^{\frac 12}=O\Big(\frac 1{\sqrt N}\Big)$$ where we used the fact that $|j_4+j_5|>N$ implies $\sup\{|j_4|, |j_5|\}>N/2$.
%\item $(l,m)=(1,3)$ (which is similar to the case $(l,m)=(2,3)$).
%In this case by Corollary \ref{orthtzv} the estimate  \eqref{A5678909998tz}
%follows by 
%\begin{equation*}
%\sum_{j_1=1}^N \Big ( \sum_{\substack
%{0< 
%|j_2|, |j_4|, |j_5|\leq N\\
%|j_4+j_5|>N}}
%\frac{1}{|j_1|^2|j_2|^2|j_4|^2|j_5|^2}
%\Big )^{\frac 12}\end{equation*}$$\leq C \sum_{j_1=1}^N \frac 1{j_1} \Big (\sum_{\substack
%{0< |j_4|, |j_5|\leq N\\
%|j_4+j_5|>N}}
%\\(j_1,j_2,j_3,j_4,j_5)\in \tilde {\mathcal A}_5^{c,j}}}
%\frac{1}{|j_4|^2|j_5|^2}
%\Big )^{\frac 12}=O\Big( \frac{\ln N}{\sqrt N}\Big)$$
%where we used the fact that $|j_4+j_5|>N$ implies $\sup\{|j_4|, |j_5|\}>N/2$.
\item $(l,m)=(1,4)$ (it is similar to $(l, m)=(1,5), (l,m)=(2,4),
(l,m)=(2,5)$).
Then by Corollary~\ref{orthtzv} and Lemma~\ref{serienew}, 
we estimate the contribution of $ {\mathcal B}_5^{c,j, (1,4)}$ to  \eqref{A5678909998tz} by
\begin{equation*}
\sum_{j_4=1}^N \Big ( \sum_{\substack
{0< |j_2|, |j_5|\leq N\\
|j_4+j_5|>N}}
\frac{1}{|j_2|^2|j_4|^4|j_5|^2}
\Big )^{\frac 12}
\end{equation*}
$$
\leq C \sum_{j_4=1}^N 
\Big (\sum_{\substack
{0< |j_5|\leq N\\
|j_4+j_5|>N}}
\frac{1}{|j_4|^4 |j_5|^2}
\Big )^{\frac 12}=
O\Big( \frac{1}{\sqrt N}\Big).
$$
\item $(l,m)=(3,4)$ (which is similar to the case
$(l,m)=(3,5)$). In this case, by Corollary \ref{orthtzv}, we estimate  
the contribution of $ {\mathcal B}_5^{c,j, (3,4)}$ to 
\eqref{A5678909998tz} by
\begin{equation*}
\sum_{j_4=1}^N \Big ( \sum_{\substack
{0< |j_1|, |j_2|, |j_5|\leq N\\
|j_4+j_5|>N}}
\frac{1}{|j_1|^2|j_2|^2|j_4|^2|j_5|^2}
\Big )^{\frac 12}\end{equation*}
\begin{equation*}
\leq C \sum_{j_4=1}^N \Big ( \sum_{\substack
{0< |j_5|\leq N\\
|j_4+j_5|>N}}
\frac{1}{|j_4|^2|j_5|^2}
\Big )^{\frac 12}=O\Big( \frac 1{\sqrt N}\Big)
\end{equation*}
where we used Lemma \ref{serienew} at the last step.
\end{itemize}
This completes the proof of  Proposition~\ref{necv}. 
\end{proof}
%%%%%%%%%%%%%%%%%%%%%%%%%%%%%%%%%%%%%%%%%%%%%%%%%%
%%%%%%%%%%%%%%%%%%%%%%%%%%%%%%%%%%%%%%%%%%%%%%%%%
\section{On the structure of conservation laws}\label{matconla}
The aim of this section is to recall some notations  (for more details see \cite{TV}).\\
Given any function $u(x)\in C^\infty(S^1)$, we define
\begin{eqnarray*}
{\mathcal P}_1(u) & = & \{\partial_x^{\alpha_1} u, 
H\partial_x^{\alpha_1} u|\alpha_1\in \N\},
\\
{\mathcal P}_2(u) & = & \{\partial_x^{\alpha_1} 
u\partial_x^{\alpha_2} u, 
(H \partial_x^{\alpha_1} u)\partial_x^{\alpha_2} u,
(H \partial_x^{\alpha_1} u)(H \partial_x^{\alpha_2} u)|\alpha_1,\alpha_2\in \N\}
\end{eqnarray*}
and in general by induction
\begin{multline*}
{\mathcal P}_n(u)=\Big \{\prod_{l=1}^k H^{i_l}p_{j_l}(u)|
i_1,...,i_k\in \{0,1\}, 
\\
\sum_{l=1}^k j_l=n, k\in \{2,...,n\}
\hbox{ and } p_{j_l}(u)\in {\mathcal P}_{j_l}(u)\Big \}
\end{multline*}
where $H$ is the Hilbert transform.
Notice that for every $n$ the simplest element belonging 
to $\mathcal P_n(u)$ has the following structure:
\begin{equation}\label{ptilde}
\prod_{i=1}^n \partial_x^{\alpha_i} u, \alpha_i\in \N.
\end{equation}
In particular we can define the map
${\mathcal P}_n(u)\ni p_n(u)\rightarrow 
\tilde p_n(u)\in {\mathcal P_n}(u)$
that associates to every $p_n(u) \in {\mathcal P}_n(u)$ the unique element
$\tilde p_n(u)\in {\mathcal P}_n(u)$ having the structure given in \eqref{ptilde}
where $\partial_x^{\alpha_1}u, \partial_x^{\alpha_2}u,
..., \partial_x^{\alpha_n}u$
are the derivatives involved in the expression of $p_n(u)$
(equivalently $\tilde p_n(u)$ is obtained from $p_n(u)$ by 
erasing all the Hilbert transforms $H$ that appear in $p_n(u)$).\\
Next, we associate to every $p_n(u)\in {\mathcal P_n}(u)$ 
two integers as follows:
$$\hbox{ if }\tilde p_n(u)=\prod_{i=1}^n \partial_x^{\alpha_i} u
\hbox{ then } |p_n(u)|:=\max_{i=1,..,n} \alpha_i
\hbox{ and }
\|p_n(u)\|:=\sum_{i=1}^n \alpha_i.$$
Given any even $k\in \N$, i.e. $k=2n$, the energy $E_{k/2}$ 
(preserved along the flow of Benjamin-Ono)
has the following structure:
\begin{multline}\label{even}E_{k/2}(u)
=  \|u\|_{\dot{H}^{n}}^2 + 
\sum_{\substack{p(u)\in {\mathcal P}_3(u) s.t. \\ 
\tilde p(u)=u\partial_x^{n-1} u \partial_x^{n}u}}
c_{2n}(p) \int p(u)dx
\\
+
\sum_{\substack{p(u)\in {\mathcal P}_{j}(u) s.t. j=3,..., 2n+2\\ 
\|p(u)\|= 2n-j+2\\ |p(u)|\leq n-1}} c_{2n}(p) \int p(u)dx 
\end{multline}
where $c_{2n}(p)\in \R$ are suitable real numbers.
For $k\in \N$ odd, i.e. $k=2n+1$, 
the energy $E_{k/2}$ has the following structure:
\begin{align}\label{odd}E_{k/2}(u)
&=  \|u\|_{\dot{H}^{n+1/2}}^2 
+ 
\sum_{\substack{p(u)\in {\mathcal P}_3(u) s.t. \\ 
\tilde p(u)=u\partial_x^{n} u \partial_x^{n}u}}
c_{2n+1}(p) \int p(u)dx
\\\nonumber
&+\sum_{\substack{p(u)\in {\mathcal P}_j(u) s.t.j=3,...,2n+3\\ 
\|p(u)\|=2n-j+3\\|p(u)|\leq n\\
\tilde p(u)\neq u\partial_x^{n} u \partial_x^{n}u}} 
c_{2n+1}(p) \int p(u)dx
\end{align}
where $c_{2n+1}(p)\in \R$ are suitable real numbers.
Observe that $E_{k/2}(u)$ can be extended from $C^\infty(S^1)$ to $H^{k/2}(S^1)$.

Next we introduce another useful notation. Given any $p(u)\in \cup_{n=2}^\infty {\mathcal P}_n(u)$ and any $N\in \N$ then we can introduce
$p_N^*(u)$ as follows. 
Let $p(u) \hbox{ be such that }$ $$\tilde p(u)=\prod_{i=1}^n 
\partial_{x}^{\alpha_i} u$$ for suitable
$0\leq \alpha_1\leq ... \leq \alpha_n$ and $\alpha_i\in \N$.
First we define $p_{i,N}^*(u)$ as the function 
obtained by $p(u)$ replacing
$\partial_x^{\alpha_i}(u)$ by $\partial_x^{\alpha_i}
(\pi_{>N} (u\partial_x u))$, i.e.
\begin{equation} 
p_{i,N}^*(u) = p(u)_{|\partial_x^{\alpha_i} u= 
\partial_x^{\alpha_i} (\pi_{>N} (u\partial_x u))}, 
\hbox{ } \forall i=1,..,n
\end{equation}
where 
$\pi_{>N}$ is the projection on the Fourier modes $n$ with $|n|>N$.
We define
$p_N^*(u)$ as follows:
$$p^*_N(u)=\sum_{i=1}^n p_{i,N}^*(u).$$
%%%%%%%%%%%%%%%%%%%%%%%%%%%%%%%%%%%%%%%%%%%%%%%%%%%%%%%%%%%%%%%%%
%%%%%%%%%%%%%%%%%%%%%%%%%%%%%%%%%%%%%%%%%%%%%%%%%%%%%%%%%%%%%%%%
\section{Proof of Theorem \ref{prcv} for $k=4$}
Recall that
\begin{eqnarray}\label{E2JK}
E_2(u) & = & \|u\|_{\dot{H}^2}^2 -\frac 54 \int  [(u_x)^2 H u_x  + 2 u u_{xx} Hu_x] dx
\\\nonumber 
& &
 + \frac 5{16}\int  [5 u^2 (u_x)^2 + u^2 (H(u_x))^2+ 2u H(u_x) H (u u_x) ] dx
\\\nonumber 
& &
+ \int [\frac 5{32} u^4 H(u_x) + \frac 5{24} u^3 H(u u_x)] 
dx+ \frac 1{48} \int u^6 dx.
\end{eqnarray}
We introduce the functions
$$G_N^{3/2}: {\rm supp}(\mu_2)\ni u_0\to \frac{d}{dt} E_{3/2
}\Big (\pi_N \Phi_t^N(u_0)\Big )_{t=0}$$
$$G_N^2: {\rm supp}(\mu_2)\ni u_0\to \frac{d}{dt} E_{2
}\Big (\pi_N \Phi_t^N(u_0)\Big )_{t=0}$$
and we recall that the functions 
$G_N^0, G_N^{1/2}, G_N^1$
are defined in Section~\ref{k=2}.
According with Proposition 5.4 in \cite{TV2}, Theorem \ref{prcv} for $k=4$,
follows by the following statement.
\begin{prop}\label{necv2}
We have
\begin{equation*}
\lim_{N\to \infty} \sum_{k=0}^4 \|G_N^{k/2}(u)\|_{L^2(d\mu_2(u))}=0.
\end{equation*}
\end{prop}
%%%%%%%%%%%%%
\begin{proof}[Proof of Proposition~\ref{necv2}]
As in the proof of Proposition~\ref{necv} it is sufficient to prove 
$$
\lim_{N\to \infty} \|G_N^{3/2}(u)\|_{L^2(d\mu_2)}+\|G_N^{2}(u)\|_{L^2(d\mu_2)}=0.
$$
\subsection{ Proof of $\|G_N^{2}(u)\|_{L^2(d\mu_2)}\rightarrow 0 \hbox{ as }
N\rightarrow \infty$}
Along the proof we denote by $u^+$ (resp. $u^-$)
the projection on positive (resp. negative) Fourier modes of $u$.
By looking at the explicit expression of $G_N^2(u)$, that follows by Proposition 3.4 in \cite{TV2} and \eqref{E2JK}, it is sufficient to prove that
\begin{equation}\label{zeroNp}
\lim_{N\to \infty} \|\int p_N^*(\pi_N u) dx\|_{L^2(d\mu_2)}=0
\end{equation}
(see section \ref{matconla} for the definition of $p_N^*(u)$)
where $p(u)$ is one of the following:
\begin{align}\label{terms} &(u_x)^2 H u_x, u u_{xx} Hu_x,
 \\\nonumber&u^2 (u_x)^2, u^2 (H(u_x))^2, u H(u_x) H (u u_x),
u^4 H(u_x), u^3 H(u u_x), u^6.
\end{align}
First, we prove \eqref{zeroNp} in the cases
$p(u)=u u_{xx} Hu_x,  p(u)=(u_x)^2 H u_x 
$ which are the most delicate terms.
The remaining terms, that involve less derivatives, will be treated at the end
via  a general argument.
\\
\\
{\em First case: proof of \eqref{zeroNp} for $p(u)=u u_{xx} Hu_x$}
\\
\\
In this case we get:
\begin{equation}\label{esrty}\int p_N^*(u)dx= \int \pi_{>N} (uu_x)u_{xx} Hu_x dx
+\int uu_{xx} \pi_{>N} H (u u_x)_x dx\end{equation}$$
+\int u Hu_x  \pi_{>N} (u u_x)_{xx}
dx=I_N(u)+II_N(u)+III_N(u).$$
Hence it is sufficient to prove that
$$\|I_N(\pi_N u)\|_{L^2 (d\mu_2)}, \|II_N(\pi_N u)\|_{L^2 (d\mu_2)}, 
\|III_N(\pi_N u)\|_{L^2 (d\mu_2)} \rightarrow 0 \hbox{ as } N\rightarrow \infty.$$
For simplicity we introduce $u_N=\pi_N u$, hence  by integration by parts and by recalling that $\pi_{>N} (u_N^+ u_N^-)=0$
we get:
$$III_N(\pi_Nu)=-\int \pi_{>N} \partial_x(u_N H\partial_x u_N)  \pi_{>N} 
\partial_x (u_N\partial_x u_N)
dx
$$
$$
= -\int \pi_{>N} \partial_x(u^+_N H\partial_x u^+)  \pi_{>N} \partial_x (u_N^-\partial_x u_N^-)dx
$$
$$
-\int \pi_{>N} \partial_x(u^-_N H\partial_x u^-_N)  \pi_{>N} \partial_x (u^+_N\partial_x u^+_N)dx
$$
%%%
$$
=  i\int \pi_{>N} \partial_x(u^+_N \partial_x u^+_N)  \pi_{>N} \partial_x (u^-_N\partial_x u^-_N) dx
$$
$$-i \int \pi_{>N} \partial_x(u^-_N \partial_x u^-_N)  
\pi_{>N} \partial_x (u^+_N\partial_x u^+_N)=0.$$
Concerning $II_N(\pi_Nu)$ (see \eqref{esrty} for definition of $II_N(u)$)
we have:
$$II_N(\pi_N u)=\int u_N \partial_x^2u_N \pi_{>N} H (u_N \partial_x^2u_N) dx +
\int u_N\partial_x^2u_N \pi_{>N} H (\partial_xu_N \partial_x u_N) dx$$
$$=\int \pi_{>N}(u^+_N\partial_x^2u^+_N) \pi_{>N} H (u^-_N\partial_x^2u^-_N) dx+\int \pi_{>N}(u_N^-\partial_x^2u_N^-) 
\pi_{>N} H (u_N^+\partial_x^2u_N^+) dx$$
$$+ \int \pi_{>N}(u^+_N\partial_x^2u^+_N) \pi_{>N} H (\partial_xu^-_N\partial_xu^-_N) dx+\int \pi_{>N}(u^-_N\partial_x^2u^-_N) \pi_{>N} H (\partial_xu^+_N\partial_xu^+_N) dx.$$
Notice that the first two terms on the r.h.s. cancel, due to the definition of $H$, and hence
it is sufficient to prove 
$$\|\int \pi_{>N}(u_N^+\partial_x^2u_N^+) \pi_{>N} H (\partial_xu_N^-\partial_xu_N^-) dx\|_{L^2(d\mu_2)}\rightarrow 0
\hbox{ as } N\rightarrow \infty$$
and 
$$\|\int \pi_{>N}(u^-_N\partial_x^2u^-_N) \pi_{>N} H (\partial_xu_N^+\partial_xu_N^+) dx\|_{L^2(d\mu_2)}
\rightarrow 0 \hbox{ as } N\rightarrow \infty$$
in order to conclude $\|II_N(\pi_Nu)\|_{L^2(d\mu_2)}\rightarrow 0$
as $N\rightarrow \infty$.
To prove the first estimate above (the second one is equivalent) 
we replace $u$ by the random vector \eqref{randomized} 
where $k=4$
and we are reduced to prove:
\begin{equation}\label{oairt}\Big \|\sum_{\substack{0<|j_1|, 
|j_2|, |j_3|, |j_4|\leq N\\
j_1, j_2>0, j_3, j_4<0\\
|j_3+j_4|>N\\(j_1,j_2,j_3,j_4)\in {\mathcal A}_4}} 
\frac{g_{j_1}g_{j_2}g_{j_3}g_{j_4}}{|j_1|^2 |j_3|
|j_4|}\Big \|_{L^2(dp)}\rightarrow 0 \hbox{ as } N\rightarrow \infty.
\end{equation}
Notice that by combining \eqref{unionA5} with the Minkowski inequality,
it is sufficient to prove \eqref{oairt} where
the condition $\vec j\in {\mathcal A}_4$
is replaced respectively by $\vec j\in \tilde {\mathcal A}_4$ and
$\vec j\in \tilde {\mathcal A}_4^c$ (where we used the notation $\vec j=(j_1, j_2, j_3, j_4)$).
In the first case (i.e. we have $\vec j\in \tilde {\mathcal A}_4$ in \eqref{oairt}) we can combine an orthogonality 
argument 
with Corollary~\ref{5,5} and the estimate follows by:
$$\sum_{\substack{0<|j_1|, 
|j_3|, |j_4|\leq N\\
|j_3+j_4|>N}} 
\frac{1}{|j_1|^4|j_3|^2|j_4|^2}=O\Big( \frac 1N\Big)
$$
where we used the fact that $|j_3+j_4|>N$ implies $\max\{|j_3|, |j_4|\}>N/2$.
In the second case (i.e. we have $\vec j\in \tilde {\mathcal A}_4^{c}$ in \eqref{oairt})
we can combine Remark~\ref{bersstr} with the Minkowski inequality
and the estimate \eqref{oairt} follows by:
$$\sum_{\substack{0<|h|, |k|\leq N\\ |h+k|>N}} 
\frac{1}{|h||k|^3}=O\Big(\frac{\ln N}{N}\Big)$$
where we used Lemma~\ref{algebrTV}.
\\

Next we prove $\|I_N(\pi_N u)\|_{L^2(d\mu_2)}\rightarrow 0$
as $N\rightarrow \infty$ (recall $I_N(u)$ is defined in \eqref{esrty}). 
This estimate is equivalent to
$$\Big \|\sum_{\substack{0<|j_1|, 
|j_2|, |j_3|, |j_4|\leq N\\
|j_1+j_2|>N\\(j_1,j_2,j_3,j_4)\in {\mathcal A}_4}} 
(sign j_4 )\frac{g_{j_1}g_{j_2}g_{j_3}g_{j_4}}{|j_1|^2|j_2|
|j_4|}\Big \|_{L^2(dp)}\rightarrow 0
\hbox{ as } N\rightarrow \infty$$
and in turn, by Minkowski inequality, it follows by
$$\sum_{\substack{0<|j_1|, |j_2|, |j_4|\leq N\\|j_1+j_2|>N}} \frac{1}{|j_1|^2|j_2||j_4|}=O\Big( \frac{\ln^2 N}{N}\Big)
$$
where we used Lemma \ref{algebrTV} at the last step.
\\
\\
{\em Second case: proof of \eqref{zeroNp} for $p(u)=u_{x}^2 Hu_x$}
\\
\\
In this case we get
\begin{equation}\label{esrty1}\int p_N^*(u) dx= 2 \int \pi_{>N} \partial_x(u\partial_x u) \partial_xu H\partial_xu dx
+\int (\partial_x u)^2 \pi_{>N} \partial_xH (u \partial_xu) dx\end{equation}$$
%+\int u Hu_x  \pi_{>N} (u\partial_x u)_{xx}
=2 \int \pi_{>N} (\partial_x u\partial_x u) \partial_xu H\partial_x u dx+
2 \int \pi_{>N} (u\partial_x^2 u) \partial_xu H\partial_xu dx
$$$$+\int (\partial_x u)^2 \pi_{>N} H (\partial_x u \partial_xu) dx
+\int (\partial_x u)^2 \pi_{>N} H (u \partial_x^2u) dx$$$$
=I_N(u)+II_N(u)+III_N(u)+IV_N(u).$$
To conclude the proof we shall show that 
$$\|I_N(\pi_N u)\|_{L^2 (d\mu_2)}, \|II_N(\pi_Nu)\|_{L^2 (d\mu_2)}
\rightarrow 0 \hbox{ as } N\rightarrow \infty$$
(by the same argument we can prove 
$\|III_N(\pi_N u)\|_{L^2 (d\mu_2)}, \|IV_N(\pi_N u)\|_{L^2 (d\mu_2)}
\rightarrow 0$).
The property $\|I_N(\pi_N u)\|_{L^2 (d\mu_2)}
\rightarrow 0 $ is equivalent to the following one  
\begin{equation}\label{oart14567}\Big \|\sum_{\substack{0<|j_1|, 
|j_2|, |j_3|, |j_4|\leq N\\
|j_1+j_2|>N\\(j_1,j_2,j_3,j_4)\in {\mathcal A}_4}} 
(sign j_4 )\frac{g_{j_1}g_{j_2}g_{j_3}g_{j_4}}{|j_1|
|j_2||j_3||j_4|}\Big \|_{L^2(dp)}\rightarrow 0
\hbox{ as } N\rightarrow \infty.\end{equation}
By combining \eqref{unionA5} with the Minkowski inequality,
it is sufficient to prove \eqref{oart14567} by assuming
that the condition $\vec j\in {\mathcal A}_4$
is replaced respectively by $\vec j\in \tilde {\mathcal A}_4$ and
$\vec j\in \tilde {\mathcal A}_4^c$ (where we used the notation $\vec j=(j_1, j_2, j_3, j_4)$).
In the first case (i.e. we have $\vec j\in \tilde {\mathcal A}_4$ in \eqref{oart14567}) we can combine an orthogonality 
argument 
with Corollary ~\ref{5,5} and the estimate follows by:
$$\sum_{\substack{0<|j_1|, 
|j_2|,|j_3|, |j_4|\leq N\\
|j_3+j_4|>N}} 
\frac{1}{|j_1|^2 |j_2|^2|j_3|^2 |j_4|^2}
\leq \sum_{\substack{0<|j_3|, |j_4|\leq N\\
|j_3+j_4|>N}} 
\frac{1}{|j_3|^2 |j_4|^2}=O\Big(\frac 1N\Big)
$$
where we have used the fact that $|j_3+j_4|>N$ implies $\max\{|j_3|, |j_4|\}>N/2$.
In the second case (i.e. we have $\vec j\in \tilde {\mathcal A}_4^{c}$ in \eqref{oart14567}),
we can combine Remark~\ref{bersstr} with the Minkowski inequality
and the estimate \eqref{oart14567} follows by:
$$\sum_{\substack{0<|h|, |k|\leq N\\ |h+k|>N}} 
\frac{1}{|h|^2|k|^2}=O\Big(\frac 1N\Big)
$$
where again we have used the fact that $|h+k|>N$ implies $\max\{|h|, |k|\}>N/2$.

Next we prove $\|II_N(\pi_N u)\|_{L^2(d\mu_2)}\rightarrow 0$
as $N\rightarrow \infty$ (recall $II_N(u)$ is defined in \eqref{esrty1}). 
This estimate follows by 
\begin{equation}\label{oart1}\Big \|\sum_{\substack{0<|j_1|, 
|j_2|, |j_3|, |j_4|\leq N\\
|j_1+j_2|>N\\(j_1,j_2,j_3,j_4)\in {\mathcal A}_4}} 
(sign j_4 )\frac{g_{j_1}g_{j_2}g_{j_3}g_{j_4}}{|j_1|^2
|j_3||j_4|}\Big \|_{L^2(dp)}\rightarrow 0
\hbox{ as } N\rightarrow \infty\end{equation}
that in turn can be proved following the proof of \eqref{oairt}. 
\\
\\
{\em Third case: proof of \eqref{zeroNp} where $p(u)$ is any of the remaining terms in \eqref{terms}}
\\
\\
It is sufficient to prove the following fact (here we use notations
introduced in Section~\ref{matconla}):
\begin{equation}\label{stronhgw}
 \hbox { if } p(u)\in {\mathcal P}_k(u), k\geq 4, |p(u)|\leq 1, \|p(u)\|= 2\end{equation}
$$\hbox{ then }
\|\int p_N^*(\pi_N u)dx\|_{L^2(d\mu_2)}\rightarrow 0 \hbox{ as }
N\rightarrow \infty.$$
We first show how to treat the simplest terms that satisfy \eqref{stronhgw}, i.e.
$p(u)=u^{k-2} \partial_x^{\nu} u\partial_x^{\mu} u$
with $0\leq \nu, \mu\leq 1$ (see the end of the proof for a comment on how to 
treat the general terms that satisfy \eqref{stronhgw}).
By looking at the explicit expression of
$p_N^*(u)$ it is sufficient to prove that
$$\|F_N(u)\|_{L^2(d\mu_2)}, \|G_N(u)\|_{L^2(d\mu_2)}, \|H_N(u)\|_{L^2(d\mu_2)}\rightarrow 0$$
where
$$F_N(u)= \int \pi_{>N} (\partial_x^\alpha u_N\partial^2_x u_N) u_N^{k-2}\partial_x^\beta u_N dx $$
$$G_N(u)=\int \pi_{>N} (\partial_x^\alpha u_N\partial_x u_N) u_N^{k-2}\partial_x^\beta u_N dx$$
$$H_N(u)= \int \pi_{>N} (u_N\partial_x u_N)u_N^{k-3} \partial_x^\alpha u_N \partial_x^\beta u_N dx$$
with $\alpha, \beta\in \{0, 1\}$ and as usual $u_N=\pi_N u$.\\
Notice that 
$\|F_N(u)\|_{L^2(d\mu_2)}\rightarrow 0$ as $N\rightarrow \infty$
is equivalent to
\begin{equation}\label{tyPLO}
\Big \|\sum_{\substack{0<|j_1|, |j_2|,
...,|j_{k+1}|\leq N\\
|j_1+j_2|>N\\\sum_{l=1}^k j_l=0}} 
\frac{g_{j_1}...
g_{j_{k+1}}}{|j_1|^{2-\alpha}|j_3|^2...|j_{k}|^2|j_{k+1}|^{2-\beta}}\Big \|_{L^2(dp)}\rightarrow 0
\hbox{ as } N\rightarrow \infty\end{equation}
that in turn due the Minkowski inequality follows (recall that $0\leq \alpha, \beta\leq 1$)
by 
$$\sum_{\substack{0<|j_1|, 
|j_3|, ..., |j_{k+1}|\leq N\\
(j_1,...,j_{k+1})\in {\mathcal A}_{k+1}}} 
\frac{1}{|j_1| |j_3|^2... |j_{k}|^2|j_{k+1}|}=O\Big(\frac{\ln^2 N}{N} \Big)$$
where we have used Lemma~\ref{algebrTV2} and we have replaced the condition
$|j_1+j_2|>N$ by $|j_3+...+j_{k+1}|>N$, since 
$j_1+...+j_{k+1}=0$.
By the same argument one can prove $\|G_N(u)\|_{L^2(d\mu_2)}\rightarrow 0$.
Concerning the proof of $\|H_N(u)\|_{L^2(d\mu_2)}\rightarrow 0$ as $N\rightarrow \infty$
it follows by 
\begin{equation*}
\Big \|\sum_{\substack{0<|j_1|, |j_2|,
...,|j_{k+1}|\leq N\\
|j_1+j_2|>N\\(j_1,...,j_{k+1})\in {\mathcal A}_{k+1}}} 
\frac{g_{j_1}...
g_{j_{k+1}}}{|j_1|^{2}|j_2||j_3|^2...|j_{k-1}|^2|j_k|^{2-\alpha} |j_{k+1}|^{2-\beta}}\Big \|_{L^2(dp)}\rightarrow 0
\hbox{ as } N\rightarrow \infty\end{equation*}
that in turn due the Minkowski inequality follows (recall that $0\leq \alpha, \beta\leq 1$)
by 
\begin{equation*}
\sum_{\substack{0<|j_1|, |j_2|,
...,|j_{k+1}|\leq N\\
|j_1+j_2|>N}} \frac{1}{|j_1|^{2}|j_2||j_3|^2...|j_{k-1}|^2|j_k| |j_{k+1}|}
=O\Big(\frac {\ln^3N}N\Big)
\end{equation*}
where we have used at the last step Lemma \ref{algebrTV}.\\

Hence we have proved \eqref{zeroNp} in the case 
$p(u)=u^{2} \partial_x^\nu u\partial_x^\mu u$. Notice that the argument above
is based on the Minkowski inequality and it
does not involve integration by parts. In the case that some $H$
appears in the expression of $p(u)$ 
(that we assume to satisfy \eqref{stronhgw}), then exactly as above we are reduced to estimate  linear combinations of multilinear products of Gaussian functions
of the type \eqref{tyPLO}, where eventually the coefficients can be affected by the 
multiplication by complex numbers of modulus one (that could come by the action of the operator $H$). However this fact does affect 
the possibility to apply the Minkowski inequality and to conclude as above.
%%%%%%%%%%%%%%%%%%%%%%%%%%%%%
\subsection{ Proof of $\|G_N^{3/2}(u)\|_{L^2(d\mu_2)}\rightarrow 0 \hbox{ as }
N\rightarrow \infty$}
In Section~\ref{odd3D} it is proved 
$$\lim_{N\to \infty} \|G_N^{3/2}(u)\|_{L^2(d\mu_{3/2}(u))}=0.$$
We claim that it implies 
$\lim_{N\to \infty} \|G_N^{3/2}(u)\|_{L^2(d\mu_{2}(u))}=0$, and hence the end of the proof of 
Proposition \ref{necv2}.
In fact, by looking at Section~\ref{odd3D}, one can deduce
that  the property $\|G_N^{3/2}(u)\|_{L^2(d\mu_{3/2}(u))}\rightarrow 0$
as $N\rightarrow \infty$, follows by
$$\|\int p_N^*(\pi_N u)\|_{L^2(d\mu_{3/2}(u))}
\rightarrow 0 \hbox{ as } N\rightarrow \infty$$
where $p(u)$ are the terms that appear in the structure of $E_{3/2}$
(see \eqref{odd}). More precisely given a suitable $p(u)$, 
the property above follows by computing explicitly $p_N^*(u)$
and by plugging in the corresponding expression the random vector defined in \eqref{randomized} 
for $k=3$. Hence everything
reduces to prove:
\begin{equation}\label{compare1}\Big \|\sum_{\substack{|j_1|,...,|j_k|\leq N\\
|j_1+j_2|>N\\(j_1,...,j_k)\in {\mathcal A}_k}} 
c(j_1,...,j_k) 
\frac{g_{j_1}...g_{j_k}}{|j_1|^{\alpha_1}...|j_k|^{\alpha_k}}
\Big \|_{L^2(dp)}\rightarrow 0 \hbox { as } N\rightarrow \infty\end{equation}
for suitable $\alpha_1,..., \alpha_k$ and where $|c(j_1,...,j_k)|=1$.
In case we are interested to prove
$$\|\int p_N^*(\pi_N u)\|_{L^2(d\mu_{2}(u))}
\rightarrow 0 \hbox{ as } N\rightarrow \infty$$
(where $p(u)$ is the same as above) then
we can argue in the same way, with the unique difference that
we plug in the expression of $p_N^*(u)$
the random vector \eqref{randomized} 
with $k=4$, hence we reduce to 
\begin{equation}\label{compare2}\Big \|\sum_{\substack{|j_1|,...,|j_k|\leq N\\
|j_1+j_2|>N\\(j_1,...,j_k)\in {\mathcal A}_k}}  
c(j_1,...,j_k) 
\frac{g_{j_1}...g_{j_k}}{|j_1|^{\beta_1}...|j_k|^{\beta_k}}
\Big \|_{L^2(dp)}\rightarrow 0 \hbox { as } N\rightarrow \infty\end{equation}
where $\beta_1,...,\beta_k$
satisfy $\beta_j\geq \alpha_j$. In fact this monotonicity is a reflection of the fact
that the coefficients in \eqref{randomized} have a stronger
decay for $k=4$ than for $k=3$.
Moreover the proof of \eqref{compare1} (that we give in section
 \ref{odd3D}) it only depends on the decay of the
coefficients $\frac 1{|j_1|^{\alpha_1}...|j_k|^{\alpha_k}}$ and not on the
oscillations of $\varphi_{j_1}...\varphi_{j_k}$ (that are exclusively exploited to perform some orthogonality arguments).
Hence the same proof works to prove \eqref{compare2},
due to the stronger decay of the coefficients. 

This completes the proof of Proposition~\ref{necv2}.
\end{proof}
%%%%%%%%%%%%%%%%%%%%%%%%%%%%%%%%%%%%%%%%%%%%%%%%%%
%%%%%%%%%%%%%%%%%%%%%%%%%%%%%%%%%%%%%%%%%%%%%%%%%%%%
\section{Estimates for $\Big \|\int p^*_N(\pi_N u) dx\Big \|_{L^2(d\mu_{m+1/2})}$}
The proof of the following lemma is inspired by
Lemma~9.1 in \cite{TV}.
We recall that $\pi_{>N}$ denotes the projector on the Fourier modes $n$ such that
$|n|>N$, and hence $\pi_N + \pi_{>N}=Id$.
\begin{lem}\label{intpar1}
Let $u(x)=\sum_{j=-N}^N c_j e^{ijx}$ with $c_0=0$, 
and $u^+(x)=\sum_{j=1}^N c_j e^{ijx}$, $u^-(x)=\sum_{j=-N}^{-1} c_j e^{ijx}$. 
Then the following identities occur:
\begin{equation}\label{mire1}
\int u (H\partial_x^m u) \partial_x^{m} H \pi_{>N} (u \partial_x u) dx=\int u \partial_x^{m} u  (\partial_x^m \pi_{>N} 
(u\partial_x u)) dx
\end{equation}
$$=\sum_{j=1}^{m}a_j \int  \pi_{>N} (u \partial_x^{m} u) 
\pi_{>N} (\partial^j_x u \partial_x^{m+1-j} u) dx $$
for suitable coefficient $a_j\in \C$. Moreover 
\begin{equation}\label{mire3}
\int u (\partial_x^m u) \partial_x^{m} H \pi_{>N} (u \partial_x u) dx
+ \int u (\partial_x^m \pi_{>N} (u \partial_x u)) (H \partial_x^{m} u) dx
=0.\end{equation}
\end{lem}
%%%%%%%%%%
\begin{proof} We first prove the second identity in \eqref{mire1}. 
We have the following identity:
$$\int u \partial_x^{m} u  (\partial_x^m \pi_{>N} 
(u\partial_x u)) dx= \frac 12 \int \pi_{>N} (u \partial_x^{m} u)  (\partial_x^m \pi_{>N} 
(\partial_x u^2)) dx $$
and by the Leibnitz rule
$$....= \int \pi_{>N} (u \partial_x^{m} u)  \partial_x (\pi_{>N} 
(u \partial^{m}_x u)) dx  $$
$$+ \sum_{j=1}^{m}a_j \int  \pi_{>N} (u \partial_x^{m} u) 
\pi_{>N} (\partial^j_x u \partial_x^{m+1-j} u) dx. $$ 
We conclude the proof 
since $\int \pi_{>N} (u \partial_x^{m} u)  \partial_x (\pi_{>N} 
(u \partial^{m}_x u)) dx=0$.\\
Next we prove the first identity in \eqref{mire1}. By using the property $\pi_{>N} (u^+ u^-)=0$ we deduce:
$$\int u (H\partial_x^m u) \partial_x^{m} H \pi_{>N} (u \partial_x u) dx=
\int \pi_{>N}  (u^+ H\partial_x^m u^+) \partial_x^{m} H \pi_{>N} (u^- \partial_x u^-) dx
$$$$
+\int \pi_{>N} (u^- H\partial_x^m u^-) \partial_x^{m} H \pi_{>N} (u^+ \partial_x u^+) dx
$$
and by definition of Hilbert transform
\begin{equation}\label{lew}
...=  \int \pi_{>N}  (u^+ (\partial_x^m u^+)) \partial_x^{m} \pi_{>N} (u^- \partial_x u^-) dx
\end{equation}$$
+\int \pi_{>N} (u^- \partial_x^m u^-) \partial_x^{m} \pi_{>N} (u^+ \partial_x u^+) dx.
$$
On the other hand by using again $\pi_{>N} (u^+ u^-)=0$ we get
$$\int u \partial_x^m u \partial_x^{m} \pi_{>N} (u \partial_x u) dx=
$$
$$\int \pi_{>N}  (u^+ \partial_x^m u^+) \partial_x^{m} \pi_{>N} (u^- \partial_x u^-) dx
+\int \pi_{>N} (u^- \partial_x^m u^-) \partial_x^{m}  \pi_{>N} (u^+ \partial_x u^+) dx.
$$
We conclude the first identity of \eqref{mire1} since the r.h.s. in the above identity, is equal to the r.h.s. in \eqref{lew}.

Next, we focus on \eqref{mire3}. 
By using again the property $\pi_{>N} (u^+u^-)=0$ we get:
$$\int u \partial_x^m u (H\partial_x^{m}  \pi_{>N} (u \partial_x u)) dx
+ \int u (H\partial_x^m u) \partial_x^{m} \pi_{>N} (u \partial_x u) dx=
$$
$$\int \pi_{>N}  (u^+ \partial_x^m u^+) \partial_x^{m} H \pi_{>N} (u^- \partial_x u^-) dx
+\int \pi_{>N} (u^- \partial_x^m u^-) \partial_x^{m} H \pi_{>N} (u^+ \partial_x u^+) dx
$$
$$+\int \pi_{>N}  (u^+ H \partial_x^m u^+) \partial_x^{m} \pi_{>N} (u^- \partial_x u^-) dx
+\int \pi_{>N} (u^- H \partial_x^m u^-) \partial_x^{m} \pi_{>N} (u^+ \partial_x u^+) dx
$$
and we can continue
$$...= i \int \pi_{>N}  (u^+ \partial_x^m u^+) \partial_x^{m}  \pi_{>N} (u^- \partial_x u^-) dx
$$$$-i\int \pi_{>N} (u^- \partial_x^m u^-) \partial_x^{m} \pi_{>N} (u^+ \partial_x u^+) dx
$$
$$-i\int \pi_{>N}  (u^+ \partial_x^m u^+) \partial_x^{m} \pi_{>N} (u^- \partial_x u^-) dx
$$$$+i\int \pi_{>N} (u^- \partial_x^m u^-) \partial_x^{m} \pi_{>N} (u^+ \partial_x u^+) dx
=0.$$
This completes the proof of Lemma~\ref{intpar1}.
\end{proof}
%%%%%%%%%%%%%%%%%%%%%
\begin{lem}\label{odd1} 
For every $m\geq 1$ we have:
\begin{equation}\label{eq:sing}
\lim_{N\rightarrow \infty} \Big \|\int p^*_N(\pi_N u) dx\Big \|_{L^2(d\mu_{m+1/2})}=0
\end{equation}
where: $$p(u)=u\partial_x^mu \partial_x^m u,
p(u)=u(H \partial_x^m u) (H\partial_x^m u) \hbox{ and } p(u)=u(H \partial_x^m u)
\partial_x^m u.$$
\end{lem}
%%%%%%
\begin{proof} We shall focus on the case $p(u)=u\partial_x^mu \partial_x^m u$
(for $p(u)=u(H \partial_x^m u) (H\partial_x^m u)$ then
it is sufficient to repeat the argument below). \\
We can write explicitly
\begin{multline}\label{ottmire}
\int p_N^*(u)= \int \pi_{>N} (u \partial_x u)(\partial_x^m u) \partial_x^{m}u
\\
+ 
2\int u \partial_x^{m}u  \partial_x^m (\pi_{>N} (u\partial_x u))\equiv I_N(u)+2II_N(u).
\end{multline}
Next, we prove that $$\|I_N(\pi_N u)\|_{L^2(d\mu_{m+1/2}(u))}\rightarrow 0
\hbox{ and } \|II_N(\pi_N u)\|_{L^2(d\mu_{m+1/2}(u))}\rightarrow 0
\hbox{ as } N\rightarrow \infty.$$
In fact the property $\|I_N(\pi_N u)\|_{L^2(d\mu_{m+1/2})}\rightarrow 0$
is equivalent to
\begin{equation}\label{A4}
\Big\|\sum_{\substack{0<|j_1|, |j_2|, |j_3|, |j_4|\leq N\\
|j_1+j_2|>N\\(j_1,j_2,j_3,j_4)\in {\mathcal A}_4}} 
\frac{g_{j_1}g_{j_2}g_{j_3}g_{j_4}}{|j_1|^{m+1/2}
|j_2|^{m-1/2}
|j_3|^{1/2}|j_4|^{1/2}}
\Big \|_{L^2(dp)}\rightarrow 0 \hbox{ as } N\rightarrow \infty.
\end{equation}
By combining \eqref{unionA5} with the Minkowski inequality,
it is sufficient to prove \eqref{A4} where
the condition $\vec j\in {\mathcal A}_4$
is replaced respectively by $\vec j\in \tilde {\mathcal A}_4$ and
$\vec j\in \tilde {\mathcal A}_4^c$ (here we used the notation $\vec j=(j_1, j_2, j_3, j_4)$).
In the first case (i.e. we have $\vec j\in \tilde {\mathcal A}_4$ in \eqref{A4}) we can combine an orthogonality 
argument 
with Corollary  \ref{5,5} and the estimate follows by:
$$
\sum_{\substack{0<|j_1|, |j_2|, |j_3|, |j_4|\leq N\\
|j_1+j_2|>N\\(j_1, j_2, j_3, j_4)\in \tilde {\mathcal A}_4}} 
\frac{1}{|j_1|^{2m+1}
|j_2|^{2m-1} |j_3| |j_4|}$$
$$\leq C \sum_{\substack{0<|j_1|, |j_2|, |j_3|, |j_4|\leq N\\
|j_1+j_2|>N}} 
\frac{1}{|j_1|^3 |j_2| |j_3| |j_4|}=O\big( \frac{\ln^3 N}{N}\Big) $$
where we we have used Lemma \ref{algebrTV}.
In the second case (i.e. we have $\vec j\in \tilde {\mathcal A}_4^{c}$ in \eqref{A4})
we can combine remark \ref{bersstr} with the Minkowski inequality
and the estimate \eqref{A4} follows by:
$$\sum_{\substack{0<|h|, |k|\leq N\\ |h+k|>N}} 
\frac{1}{|h|^m|k|^{m+1}}=O\Big(\frac{\ln N}N\Big)$$
where we used Lemma \ref{algebrTV} in conjunction with the fact that $m\geq 1$.\\
Next we prove $\|II_N(\pi_N u)\|_{L^2(d\mu_{m+1/2})}\rightarrow 0$
(where $II_N(u)$ is defined in \eqref{ottmire}).
By \eqref{mire1} in Lemma \ref{intpar1} we are reduced to prove
$
\left \|II_N^j(\pi_Nu) \right \|_{L^2(d\mu_{m+1/2})}
\rightarrow 0 \hbox{ as } N\rightarrow \infty
$
where 
$$II_N^j(u)=\int \pi_{>N} (u
\partial^{m}_xu)
\pi_{>N} (\partial_x^j u \partial_x^{m-j+1} 
u) dx, \hbox{ } j=1,...,m.$$ 
Hence we are reduced to 
prove for $j=1,...,m$ that
\begin{equation}\label{A423}\Big \|\sum_{\substack{0<|j_1|, 
|j_2|, |j_3|, |j_4|\leq N\\
|j_1+j_2|>N\\(j_1,j_2,j_3,j_4)\in {\mathcal A}_4}} 
\frac{g_{j_1}g_{j_2}g_{j_3}g_{j_4}}{|j_1|^{m+1/2}|j_2|^{1/2}
|j_3|^{m+1/2-j}|j_4|^{j-1/2}}\Big \|_{L^2(dp)}\rightarrow 0
\end{equation}
whose proof follows by a similar argument used along the proof of 
\eqref{A4}.\\
\\
Concerning the proof of \eqref{eq:sing} in the case $p(u)=u(H \partial_x^m u) \partial_x^m u$,
notice that
we have 
\begin{multline*}
\int p_N^*(u)= \int\pi_{>N} (u \partial_x u)(H \partial_x^m u) \partial_x^{m}u
+ 
\int u \partial_x^{m}u  \partial_x^m (\pi_{>N} H (u\partial_x u))
\\
+\int u (H \partial_x^{m}u)  \partial_x^m (\pi_{>N}(u\partial_x u))
\equiv I_N(u)+II_N(u)+III_N(u).
\end{multline*}
By \eqref{mire3} we get $II_N (\pi_N u)+III_N(\pi_N u)=0$,
hence it is sufficient to show that
$\|I_N(\pi_N u)\|_{L^2(d\mu_{m+1/2})}\rightarrow 0$.
Arguing as above this estimate follows by 
\begin{equation}\label{A4strmer}
\Big\|\sum_{\substack{0<|j_1|, |j_2|, |j_3|, |j_4|\leq N\\
|j_1+j_2|>N\\(j_1,j_2,j_3,j_4)\in {\mathcal A}_4}} 
({\rm sign }(j_3)) \frac{g_{j_1}g_{j_2}g_{j_3}g_{j_4}}{|j_1|^{m+1/2}
|j_2|^{m-1/2}
|j_3|^{1/2}|j_4|^{1/2}}
\Big \|_{L^2(dp)}\rightarrow 0 
\end{equation}
as  $N\rightarrow \infty$
and it can be proved following the same argument used above to estimate
\eqref{A4}.
This completes the proof of Lemma~\ref{odd1}. 
\end{proof}
%%%%%%%%%%%%%%%%%%%%%%%%%%%%%%%%%%%%%%%
\begin{lem}\label{odd3} For every $m\geq 1$ we have:
\begin{equation*}
\lim_{N\rightarrow \infty} \Big \|\int p^*_N(\pi_N u) dx\Big \|_{L^2(d\mu_{m+1/2}(u))}=0
\end{equation*}
where: 
\begin{equation}\label{asss}
p(u)\in {\mathcal P}_3(u), |p(u)|\leq m, \|p(u)\|=2m, \tilde p(u)\neq u\partial_x^mu
\partial_x^m u.
\end{equation}
\end{lem}
\begin{proof}
We consider some specific $p(u)$ that satisfy the assumptions:
\begin{equation}\label{basicstruct4563}
p(u)=\partial_x^{\alpha_1} u \partial_x^{\alpha_2} u \partial_x^{\alpha_3} u
\end{equation}
\begin{equation}\label{constrained3} \sum_{i=1}^3 \alpha_i=2m
\hbox{ and } \alpha_1, \alpha_2<m, \alpha_3\leq m.\end{equation}
Notice that we have chosen $p(u)$ such that the Hilbert transform
$H$ is not involved. By looking at the argument 
below and arguing as at the end of the proof of Proposition
\ref{necv2}, one can deduce that indeed by the same technique
one can treat any $p(u)$ that satisfies the assumptions \eqref{asss}

In the context of \eqref{basicstruct4563}, we have
\begin{multline}\label{SeE}
\int p_N^*(u) = \int \pi_{>N}\partial_x^{\alpha_1} (u\partial_x u)
\partial_x^{\alpha_2} u \partial_x^{\alpha_3} u
+
\int \pi_{>N}\partial_x^{\alpha_2} (u\partial_x u) \partial_x^{\alpha_1}u \partial_x^{\alpha_3} u
\\
+\int \pi_{>N}\partial_x^{\alpha_3} (u\partial_x u) \partial_x^{\alpha_1}u \partial_x^{\alpha_2} u 
=I_N(u)+II_N(u)+III_N(u).
\end{multline}
It is sufficient to prove that 
$$\|I_N(\pi_N u)\|_{L^2(d\mu_{m+1/2})}, \|II_N(\pi_N u)\|_{L^2(d\mu_{m+1/2})}
\|III_N(\pi_N u)\|_{L^2(d\mu_{m+1/2})}\rightarrow 0$$ as 
$N\rightarrow \infty$. \\
\\
We shall focus first on 
$\|III_N(\pi_N u)\|_{L^2(d\mu_{m+1/2})}\rightarrow 0$ as 
$N\rightarrow \infty$, 
and hence by the Leibnitz rule we 
are reduced to prove: 
$\|III_N^j(\pi_N u)\|_{L^2(d\mu_{m+1/2})}\rightarrow 0$ where
$$III_N^j(u)= \int \pi_{>N} (\partial_x^j u\partial_x^{k} u)\partial_x^{\alpha_1} u \partial_x^{\alpha_2} u 
dx
$$ 
with
\begin{equation*}
%\label{caos}
j+k=\alpha_3+1, 0\leq j,k \leq \alpha_3+1\leq m+1.
\end{equation*}
In fact it is equivalent to
\begin{equation}\label{A56783}\Big \|\sum_{\substack{0<|j_1|, 
|j_2|, |j_3|, |j_4|\leq N\\
|j_3+j_4|>N\\(j_1,j_2,j_3,j_4)\in {\mathcal A}_4}} 
b_{j_1j_2j_3j_4}
g_{j_1}g_{j_2}g_{j_3}g_{j_4}
\Big \|_{L^2(dp)}\rightarrow 0\end{equation}
where:
$$b_{j_1j_2j_3j_4}=\frac 1{|j_1|^{m+1/2-j}
|j_2|^{m+1/2-k} |j_3|^{m+1/2-\alpha_1} 
|j_4|^{m+1/2-\alpha_2}}.$$
Hence by \eqref{constrained3} we get
\begin{equation}\label{basiKOL}b_{j_1j_2j_3j_4}\leq \frac 1{|j_1|^{m+1/2-j}
|j_2|^{m+1/2-k} |j_3|^{3/2} 
|j_4|^{3/2}}\end{equation}
Following the same argument used to deduce \eqref{A4} it is sufficient 
to prove 
 \eqref{A56783} where the condition $\vec j\in {\mathcal A}_4$
replaced by $\vec j\in \tilde {\mathcal A}_4$ and
$\vec j\in \tilde {\mathcal A}_4^c$ (where we used the notation $\vec j=(j_1, j_2, j_3, j_4)$).\\
\\
In the first case  (i.e. we get $\vec j\in \tilde {\mathcal A}_4$
in \eqref{A56783})
we can combine Corollary \ref{5,5}
with an orthogonality argument 
and we are reduced to show
\begin{equation}\label{estimatepro3}\sum_{\substack{0<|j_1|, 
|j_2|, |j_3|, |j_4|, |j_5|\leq N\\
|j_3+j_4|>N
\\(j_1,j_2,j_3,j_4)\in \tilde{\mathcal A}_4}}
\frac 1{|j_1|^{2m+1-2j}
|j_2|^{2m+1-2k} |j_3|^{3} |j_4|^{3}} \rightarrow 0 \hbox{ as } N\rightarrow \infty\end{equation}
where $j, k$ satisfy $0\leq j,k\leq m+1$, $j+k\leq m+1$.
By symmetry we can assume that $0\leq k\leq j$. 
%\begin{equation}\label{smoothcontr} j+k=\alpha_3+1, 0\leq j\leq \alpha_3
%\end{equation}
Next we consider several cases.
\begin{itemize}
\item $j=m$. Then $k\leq 1$  and we get
$$\frac 1{|j_1|^{2m+1-2j}
|j_2|^{2m+1-2k} |j_3|^{3} |j_4|^{3}}\leq \frac 1{|j_1||j_2|^{2m-1}  |j_3|^{3} |j_4|^{3}}$$
and hence the l.h.s. of \eqref{estimatepro3} is $O\Big(\frac{\ln^2N} {N^2}\Big)$,
where we used the fact that $|j_3+j_4|>N$ implies $\max\{|j_3|, |j_4|\}>N/2$.
\item $j=m+1$. Then $k=0$ and hence
$$\frac 1{|j_1|^{2m+1-2j}
|j_2|^{2m+1-2k} |j_3|^{3} |j_4|^{3}}\leq \frac {|j_1|}{|j_2|^3  |j_3|^{3} |j_4|^{3}}.$$
By combining the fact $|j_3+j_4|>N$ implies $\max\{|j_3|, |j_4|\}>N/2$ with Lemma \ref{sersaut}, we get that the l.h.s. in \eqref{estimatepro3} is
$O\big( \frac{\ln N}{N}\big)$. 
\item $j\leq m-1$. Then $k\leq m-1$ and we deduce
$$\frac 1{|j_1|^{2m+1-2j}
|j_2|^{2m+1-2k} |j_3|^{3} |j_4|^{3}}\leq \frac {1}{|j_1|^3  |j_2|^3|j_3|^{3} |j_4|^{3}}.$$
Hence we get that the l.h.s. in \eqref{estimatepro3} is $O\Big(\frac{1}{N} \Big)$,
where we used the fact that $|j_3+j_4|>N$ implies $\max\{|j_3|, |j_4|\}>N/2$.
\end{itemize}
Next we focus on the proof of \eqref{A56783} in the case that 
the constraint $\vec j\in {\mathcal A}_4$ is replaced by 
$\vec j\in \tilde {\mathcal A}_4^c$, where as usual $\vec j=(j_1, j_2, j_3, j_4)$.
By combining Remark~\ref{bersstr} with the Minkowski inequality,
and by recalling \eqref{basiKOL},
we deduce that the estimate \eqref{A56783} (where the constraint $\vec j\in {\mathcal A}_4$ is replaced by  $\vec j\in \tilde {\mathcal A}_4^c$) follows by:
\begin{equation}\label{soprTY}
\sum_{\substack{0<|h|, |k|\leq N\\ |h+l|>N}} 
\frac{1}{ |h|^{m+2-j} |l|^{m+2-k} }\rightarrow 0 \hbox{ as } N\rightarrow \infty
\end{equation}
where $j, k$ satisfy $0\leq j,k\leq m+1$, $j+k\leq m+1$. 
Again by symmetry, we may assume $k\leq j$.
As a consequence 
$$
\frac{1}{ |h|^{m+2-j} |l|^{m+2-k} }\leq \frac{1}{|h||l|^2}
$$
and \eqref{soprTY}
follows by Lemma~\ref{algebrTV}.
\\
\\
Concerning the proof of $$\|I_N(\pi_N u)\|_{L^2(d\mu_{m+1/2})}, \|II_N(\pi_N u)\|_{L^2(d\mu_{m+1/2})}\rightarrow 0$$
(see \eqref{SeE} for the definition of $I_N(u)$ and $II_N(u)$)
it is easy to see that the estimates are equivalent, hence let's focus on 
$\|I_N(\pi_N u)\|_{L^2(d\mu_{m+1/2})}\rightarrow 0$.
 By the Leibnitz rule we 
are reduced to prove: 
$\|I_N^j(\pi_N u)\|\rightarrow 0$ as $N\rightarrow \infty$ where
$$I_N^j(u)= \int \pi_{>N} (\partial_x^j u\partial_x^{k} u)\partial_x^{\alpha_2} u \partial_x^{\alpha_3} u 
dx$$ 
\begin{equation}\label{caosT}
j+k=\alpha_1+1, 0\leq j,k \leq \alpha_1+1\leq m.
\end{equation}
In fact it is equivalent to
\begin{equation}\label{A56783T}\Big \|\sum_{\substack{0<|j_1|, 
|j_2|, |j_3|, |j_4|\leq N\\
|j_3+j_4|>N\\(j_1,j_2,j_3,j_4)\in {\mathcal A}_4}} 
c_{j_1j_2j_3j_4}
g_{j_1}g_{j_2}g_{j_3}g_{j_4}
\Big \|_{L^2(dp)}\rightarrow 0\end{equation}
where
$$c_{j_1j_2j_3j_4}=\frac 1{|j_1|^{m+1/2-j}
|j_2|^{m+1/2-k} |j_3|^{m+1/2-\alpha_2} 
|j_4|^{m+1/2-\alpha_3}}.$$
Hence by \eqref{caosT} and by recalling that $\alpha_1, \alpha_2<m$
(see \eqref{constrained3}), we get easily
\begin{equation}\label{basiKOLT}c_{j_1j_2j_3j_4}\leq \frac 1{|j_1|^{3/2}
|j_2|^{1/2} |j_3|^{3/2} 
|j_4|^{1/2}}.\end{equation}
Following the same argument as above it is sufficient 
to prove 
 \eqref{A56783T} under the condition $\vec j\in {\mathcal A}_4$
replaced by $\vec j\in \tilde {\mathcal A}_4$ and
$\vec j\in \tilde {\mathcal A}_4^c$ (where we used the notation $\vec j=(j_1, j_2, j_3, j_4)$).
In the first case  (i.e. we get $\vec j\in \tilde {\mathcal A}_4$
in \eqref{A56783T})
we can combine Corollary \ref{5,5}
with an orthogonality argument 
and we are reduced to show
$$\sum_{\substack{0<|j_1|, 
|j_2|, |j_3|, |j_4|, |j_5|\leq N\\
|j_3+j_4|>N}}
\frac 1{|j_1|^{3}
|j_2| |j_3|^{3} |j_4|}=O\Big(\frac{\ln^2N}{N}\Big)$$
where we have used Lemma \ref{algebrTV} at the last step.
In the case that 
the constraint $\vec j\in {\mathcal A}_4$ is replaced by 
$\vec j\in \tilde {\mathcal A}_4^c$ in \eqref{A56783T} then,
by combining Remark~\ref{bersstr} with the Minkowski inequality,
and by recalling \eqref{basiKOLT},
we deduce that the estimate \eqref{A56783T}  follows by:
\begin{equation}\label{soprTYT}
\sum_{\substack{0<|h|, |k|\leq N\\ |h+l|>N}} 
\frac{1}{ |h|^{3} |l| }+ \sum_{\substack{0<|h|, |k|\leq N\\ |h+l|>N}} 
\frac{1}{ |h|^{2} |l|^2 }=O\Big(\frac{\ln N}N \Big)
\end{equation}
where we used Lemma \ref{algebrTV} to estimate the first term, and the fact that
$|h+k|>N$ implies $\max\{|h|, |k|\}>N/2$, to estimate the second term.
%Next we comment about the cases that $p(u)$ has not the structure given in \eqref{basicstruct4563}. In fact it could happen that some Hilbert transforms $H$ appear somewhere in the expression. Notice that in the case  \eqref{basicstruct4563} we have reduced our analysis to prove estimates of the type \eqref{A56783}. In particular no integration by parts have been involved in order to reduceour problem to  \eqref{A56783},but we have simply plugged in the expression of$p_N^*(\pi_N u)$ the randomized vector. In case that we repeat the same argumentfor a general $p(u)$ (where some $H$ could appear) then we can reduce at the same kind of estimates as in \eqref{A56783}, but with the coefficients$b_{j_1j_2j_3j_4}$ eventually modified by  $d_{j_1j_2j3j_4} b_{j_1j_2j_3j_4}$ with $|d_{j_1j_2j_3j_4}|=1$. As a consequence, the subsequent analysis that we did above can be repeated,since it is based on orthogonality arguments and Minkowski inequality that are not affected  by the modification of the coefficients with numbers of modulus equal to one.
This completes the proof of Lemma~\ref{odd3}.
\end{proof}
%%%%%%%%%%%%%%%%%%%%%%%%%%%%%%%%%%%%
\begin{lem}\label{odd2} For every $m\geq 1$ we have
\begin{equation*}
\lim_{N\rightarrow \infty} \Big \|\int p^*_N(\pi_N u) dx\Big \|_{L^2(d\mu_{m+1/2})}=0
\end{equation*}
where: $$p(u)\in {\mathcal P}_4(u), |p(u)|\leq m, \|p(u)\|=2m-1.$$
\end{lem}
\begin{proof}
Again by the argument at the end of Proposition~\ref{necv2},  we can obtain that it suffices to treat the case
\begin{equation}\label{basicstruct456}
p(u)=\partial_x^{\alpha_1} u \partial_x^{\alpha_2} u \partial_x^{\alpha_3} u
\partial_x^{\alpha_4} u\end{equation}
\begin{equation}\label{constrained5} \sum_{i=1}^4 \alpha_i=2m-1
\hbox{ and } \max_{i=1,2,3,4} \alpha_i=\alpha_4\leq m .\end{equation}
By explicit expression of $p^*_N(u)$ we get
\begin{multline}\label{cOnGhT}
\int p_N^*(u) dx= \int \pi_{>N}\partial_x^{\alpha_1} (u\partial_x u)
\partial_x^{\alpha_2} u \partial_x^{\alpha_3} u \partial_x^{\alpha_4} u
dx 
\\
+
\int \pi_{>N}\partial_x^{\alpha_2} (u\partial_x u) \partial_x^{\alpha_1}u \partial_x^{\alpha_3} u
\partial_x^{\alpha_4} u dx
+\int \pi_{>N}\partial_x^{\alpha_3} (u\partial_x u) \partial_x^{\alpha_1}u \partial_x^{\alpha_2} u
\partial_x^{\alpha_4} u dx 
\\
+ \int \pi_{>N}\partial_x^{\alpha_4} (u\partial_x u) \partial_x^{\alpha_1}u \partial_x^{\alpha_2} u
\partial_x^{\alpha_3} u dx 
\\
=I_N(u)+II_N(u)+III_N(u)+IV_N(u).
\end{multline}
It is sufficient to prove that 
\begin{multline*}
\|I_N(\pi_N u)\|_{L^2(d\mu_{m+1/2})}, \|II_N(\pi_N u)\|_{L^2(d\mu_{m+1/2})},
\\
\|III_N(\pi_N u)\|_{L^2(d\mu_{m+1/2})},\|IV_N(\pi_N u)\|_{L^2(d\mu_{m+1/2})}\rightarrow 0\hbox{ as }
N\rightarrow \infty.
\end{multline*}
In fact by symmetry the first three terms above are equivalent, hence we shall focus 
on the proof of
$\|I_N(\pi_N u)\|_{L^2(d\mu_{m+1/2})}\rightarrow 0$ and
$\|IV_N(\pi_N u)\|_{L^2(d\mu_{m+1/2})}\rightarrow 0$ as 
$N\rightarrow \infty$. \\
\\
{\bf Estimates for $\|IV_N(\pi_N u)\|_{L^2(d\mu_{m+1/2})}$}
\\
\\
Notice that by the Leibnitz rule it is sufficient to prove  
$\|IV_N^j(\pi_N u)\|\rightarrow 0$ as $N\rightarrow \infty$ where
$$IV_N^j(u)= \int \pi_{>N} (\partial_x^j u\partial_x^{k} u)\partial_x^{\alpha_1} u \partial_x^{\alpha_2} u \partial_x^{\alpha_3} u
dx$$
\begin{equation}\label{conne5}
j+k=\alpha_4+1\leq m+1,\quad
j,k\geq 0.
\end{equation}
It is sufficient to prove that
\begin{equation}\label{A5678}\Big \|\sum_{\substack{0<|j_1|, 
|j_2|, |j_3|, |j_4|, |j_5|\leq N\\
|j_4+j_5|>N\\(j_1,j_2,j_3,j_4,j_5)\in {\mathcal A}_5}} 
b_{j_1j_2j_3j_4j_5}
g_{j_1}g_{j_2}g_{j_3}g_{j_4}g_{j_5}
\Big \|_{L^2(dp)}\rightarrow 0\end{equation}
where
\begin{equation}\label{bjjjj}b_{j_1j_2j_3j_4j_5}=\frac 1{|j_1|^{m+1/2-\alpha_1}
|j_2|^{m+ 1/2-\alpha_2} |j_3|^{m+1/2-\alpha_3} |j_4|^{m+1/2-j}|j_5|^{m+1/2-k}}.
\end{equation}
Due to \eqref{constrained5} we get
\begin{equation}\label{conSTR}
b_{j_1j_2j_3j_4j_5}\leq
\frac 1{|j_1|^{3/2}
|j_2|^{3/2} |j_3|^{3/2} |j_4|^{m+1/2-j}|j_5|^{m+1/2-k}}.
\end{equation}
Following the same argument used to deduce \eqref{A4}, it is sufficient 
to prove 
 \eqref{A5678} with the condition $\vec j\in {\mathcal A}_5$
replaced respectively by $\vec j\in \tilde {\mathcal A}_5$ and
$\vec j\in \tilde {\mathcal A}_5^c$ (where we used the notation $\vec j=(j_1, j_2, j_3, j_4, j_5)$).
In the first case  (i.e. when $\vec j\in \tilde {\mathcal A}_5$
in \eqref{A5678})
we can combine Corollary \ref{5,5}
with an orthogonality argument 
and by \eqref{conSTR},  we are reduced to show:
\begin{equation}\label{estimatepro59}\sum_{\substack{0<|j_1|, 
|j_2|, |j_3|, |j_4|, |j_5|\leq N\\
|j_4+j_5|>N\\(j_1,j_2,j_3,j_4,j_5)\in {\mathcal A}_5}} \frac 1{|j_1|^{3}
|j_2|^{3} |j_3|^{3} |j_4|^{2m+1-2j}|j_5|^{2m+1-2k}}
\rightarrow 0 \hbox{ as } N\rightarrow \infty\end{equation}
where $j,k$ satisfy \eqref{conne5}.
By symmetry, we can assume that $j\leq k$.
Next we consider several cases.
\begin{itemize}
\item $k=m$.
Then in this case by \eqref{conne5} we get $j\leq 1$ and hence \eqref{estimatepro59}
follows by 
$$\sum_{\substack{0<|j_1|, 
|j_2|, |j_3|, |j_4|, |j_5|\leq N\\
|j_1+j_2+j_3|>N}} \frac 1{|j_1|^{3}
|j_2|^{3} |j_3|^{3} |j_4||j_5|}=O\Big(\frac{\ln^2 N}{N^2}\Big)
$$
where we used the fact that $|j_1+j_2+j_3|>N$ implies $\max\{|j_1|,|j_2|, |j_3|\}>N/3$.
\item $k=m+1$. Then $j=0$ and hence \eqref{estimatepro59}
follows by 
$$\sum_{\substack{0<|j_1|, 
|j_2|, |j_3|, |j_4|, |j_5|\leq N\\
|j_4+j_5|>N\\|j_1+j_2+j_3|>N}} \frac {|j_5|}{|j_1|^{3}
|j_2|^{3} |j_3|^{3} |j_4|^3}=O\Big(\frac{\ln N}{N}\Big)
$$
where we used Lemma \ref{sersaut} in conjunction with the fact that 
$|j_1+j_2+j_3|>N$ implies $\max\{|j_1|,|j_2|, |j_3|\}>N/3$.
\item $k\leq m-1$. Then $j\leq m-1$ and therefore \eqref{estimatepro59}
follows by 
$$\sum_{\substack{0<|j_1|, 
|j_2|, |j_3|, |j_4|, |j_5|\leq N\\
|j_1+j_2+j_3|>N}} \frac 1{|j_1|^{3}
|j_2|^{3} |j_3|^{3} |j_4|^3|j_5|}=O\Big(\frac{\ln N}{N^2}\Big)
$$
where we used the fact that 
$|j_1+j_2+j_3|>N$ implies $\max\{|j_1|,|j_2|, |j_3|\}>N/3$.
\end{itemize}
Next we focus on the proof of \eqref{A5678} in the case that 
the constraint $\vec j\in {\mathcal A}_5$ is replaced by 
$\vec j\in \tilde {\mathcal A}_5^c$, where as usual $\vec j=(j_1, j_2, j_3, j_4, j_5)$.
Notice that by \eqref{disjAc5} (where the union is disjoint for $n=5$ by Remark~\ref{rem5})
and by Minkowski inequality, it is sufficient to prove that
\begin{equation*}
%\label{A5678909998}
\sum_{j=1}^N \Big \| \sum_{\substack
{0<|j_1|, 
|j_2|, |j_3|, |j_4|, |j_5|\leq N\\
|j_4+j_5|>N\\(j_1,j_2,j_3,j_4,j_5)\in \tilde {\mathcal A}_5^{c,j}}}
b_{j_1j_2j_3j_4j_5}
g_{j_1}g_{j_2}g_{j_3}g_{j_4}g_{j_5} 
\Big\|_{L^2(dp)}
\rightarrow 0 
\end{equation*}
(see \eqref{bjjjj} for definition of $b_{j_1j_2j_3j_4j_5}$),
that in turn due to  \eqref{dijfinfat}
follows by 
\begin{equation}\label{A5678909998tzv}
\sum_{j=1}^N \Big \| \sum_{\substack
{0<|j_1|, 
|j_2|, |j_3|, |j_4|, |j_5|\leq N\\
|j_4+j_5|>N\\(j_1,j_2,j_3,j_4,j_5)\in {\mathcal B}_5^{c,j, (l,q)}}}
b_{j_1j_2j_3j_4j_5}g_{j_1}g_{j_2}g_{j_3}g_{j_4}g_{j_5}  \Big\|_{L^2(dp)}
\rightarrow 0 \end{equation}
with $(l, q)=(1,2), (1,3), (1,4) ,(1,5), (2,3) ,(2,4) ,(2,5), (3,4), (3,5)$.
Indeed we have excluded the possibility $(l,q)=(4,5)$ since
for every $(j_1, j_2, j_3, j_4, j_5)\in {\mathcal B}_5^{c,j, (4,5)}$
we get $|j_4+j_5|=|j-j|=0$ which is in contradiction with
the constrained $|j_4+j_5|>N$ that appears in 
\eqref{A5678909998tzv}. In fact arguing as we did along the proof of Proposition~\ref{necv} (see below \eqref{A5678909998tz}) we can also exclude
also $(l,q)=(1,2),(1,3), (2,3)$.\\
Hence we remain with the cases:
$(l,q)=(1,4), (1,5), (2,4), (2,5), (3,4), (3,5).$
By the symmetry of the roles of $\alpha_1,\alpha_2, \alpha_3$ and by a symmetry of the roles of $j,k$, we obtain that, it suffices to consider the case 
$(l,q)=(1,4)$. 
 
By using the orthogonality stated in Corollary~\ref{orthtzv}, the estimate \eqref{A5678909998tzv}
follows by 
\begin{equation}\label{A5678909998tzvlm12rs}
\sum_{j_4=1}^N \Big (\sum_{\substack
{0<|j_2|, |j_3|, |j_5|\leq N\\
j_2+j_3+j_5=0\\|j_2+j_3-j_4|>N
}}
|b_{j_4j_2j_3j_4j_5}|^2 \Big)^\frac 12
\rightarrow 0 \hbox{ as } N\rightarrow \infty \end{equation}
where used the fact that $j_1=-j_4$ to get the constraint $|j_2+j_3-j_4|>N
$. We now consider three cases according to the values of $j$ and $k$
\begin{itemize}
\item $j=m+1$. Then $k=0$ and thus
$|b_{j_4j_2j_3j_4j_5}|^2\leq \frac {1}{
|j_2|^{3} |j_3|^{3} |j_4|^{2}|j_5|^3}$. 
Hence the l.h.s. in \eqref{A5678909998tzvlm12rs} 
can be estimated by
\begin{equation}\label{finrip}
\sum_{j_4=1}^N \frac{1}{|j_4|}\Big (\sum_{\substack
{0<|j_2|, |j_3|, |j_5|\leq N\\
|j_4+j_5|>N
}} \frac {1}{
|j_2|^{3} |j_3|^{3} |j_5|^3} \Big)^\frac 12=O\Big( \frac{1}{
\sqrt N}\Big)
\end{equation}
where we used Lemma~\ref{serienew}.
\item $k=m+1$. Then $j=0$ and hence
$|b_{j_4j_2j_3j_4j_5}|^2\leq \frac {|j_5|}{
|j_2|^{3} |j_3|^{3} |j_4|^{6}}$. In this case
the l.h.s. in \eqref{A5678909998tzvlm12rs} can be estimated by
$$
\sum_{j_4=1}^N \frac{1}{|j_4|^3}\Big (\sum_{\substack
{0<|j_2|, |j_3|, |j_5|\leq N\\
j_2+j_3+j_5=0
\\
|j_2+j_3-j_4|>N
}} \frac {|j_5|}{
|j_2|^{3} |j_3|^{3} } \Big)^\frac 12
=O\Big( \frac{1}{\sqrt N}\Big),
$$
where we have used that $|j_5|\leq N$ and the summation on $j_5$ can be ignored together with the bound 
$\max(|j_2|,|j_3|,|j_4|)\geq N/3$.
%%%%%%
\item $j,k\leq m$.
In this case 
$|b_{j_4j_2j_3j_4j_5}|^2\leq \frac {1}{
|j_2|^{3} |j_3|^3 |j_4|^{4} |j_5|}$. 
Hence the l.h.s. in \eqref{A5678909998tzvlm12rs} 
can be estimated by
\begin{equation}\label{finripc}\sum_{j_4=1}^N \frac 1{|j_4|^2}\Big (\sum_{\substack
{0<|j_2|, |j_3|, |j_5|\leq N
\\|j_2-j_3+j_4|>N
%\\|j_2+j_3-j_4|>N
}} \frac {1}{|j_2|^{3} |j_3|^3 ||j_5|} \Big)^\frac 12= O \Big (\frac{\ln N}{ N}\Big )\end{equation}
where we have used that $\max(|j_2|,|j_3|,|j_4|)\geq N/3$.
\end{itemize}
%%%%%%%%%%%%%%%%%%%%%%%%%%%%%%%%%%%%%%%%%%%%
%%%%%%%%%%%%%%%%%%%%%%%%%%%%%%%%%%%%%%%%%%%%%
{\bf Estimates for $\|I_N(\pi_N u)\|_{L^2(d\mu_{m+1/2})}$}
\\
\\
Recall that $I_N(u)$ is defined in \eqref{cOnGhT}.
Notice that by the Leibnitz rule it is sufficient to prove  
$\|I_N^j(\pi_N u)\|\rightarrow 0$ as $N\rightarrow \infty$ where
$$I_N^j(u)= \int \pi_{>N} (\partial_x^j u\partial_x^{k} u)\partial_x^{\alpha_2} u \partial_x^{\alpha_3} u \partial_x^{\alpha_4} u
dx$$
\begin{equation}\label{conne5I}
j+k=\alpha_1+1\leq m,\quad
j,k \geq 0.
\end{equation}
It is sufficient to prove that
\begin{equation}\label{A5678I}\Big \|\sum_{\substack{0<|j_1|, 
|j_2|, |j_3|, |j_4|, |j_5|\leq N\\
|j_1+j_2|>N\\(j_1,j_2,j_3,j_4,j_5)\in {\mathcal A}_5}} 
b_{j_1j_2j_3j_4j_5}
g_{j_1}g_{j_2}g_{j_3}g_{j_4}g_{j_5}
\Big \|_{L^2(dp)}\rightarrow 0\end{equation}
where
\begin{equation}\label{bjjjjI}b_{j_1j_2j_3j_4j_5}=\frac 1{|j_1|^{m+1/2-j}
|j_2|^{m+ 1/2-k} |j_3|^{m+1/2-\alpha_2} |j_4|^{m+1/2-\alpha_3}|j_5|^{m+1/2-\alpha_4}}.
\end{equation}
Hence due to \eqref{constrained5} we get
\begin{equation}\label{conSTRI}
b_{j_1j_2j_3j_4j_5}\leq
\frac 1{|j_1|^{m+1/2-j}|j_2|^{m+1/2-k} |j_3|^{3/2}
|j_4|^{3/2} |j_5|^{m+1/2-\alpha_4}}.
\end{equation}
Notice that in the case $\alpha_4<m$ then the estimate above reduces to
\begin{equation}\label{conSTRIII}
b_{j_1j_2j_3j_4j_5}\leq
\frac 1{|j_1|^{m+1/2-j}|j_2|^{m+1/2-k} |j_3|^{3/2}
|j_4|^{3/2} |j_5|^{3/2}}
\end{equation}
which actually is equivalent to \eqref{conSTR}. Since moreover 
the condition \eqref{conne5I} is more restrictive than \eqref{conne5}
(since $\alpha_1<m$ by assumption, while in \eqref{conne5} $\alpha_4$
could be equal to $m$) we deduce that the estimate for $I_N$ can be done exactly as we did for $IV_N$ in the case $\alpha_4<m$.\\
\\
Next we focus on the case $\alpha_4=m$. 
By a symmetry argument, we can assume that $j\leq k$ and therefore by \eqref{conne5I} we get 
$j\leq m-1$, $k\leq m$. Thus by \eqref{conSTRI} we deduce
the following estimate:
\begin{equation}\label{AbCf}b_{j_1j_2j_3j_4j_5}\leq
\frac 1{|j_1|^{3/2}|j_2|^{1/2} |j_3|^{3/2}
|j_4|^{3/2} |j_5|^{1/2}}.
\end{equation}
%\begin{equation}b_{j_1j_2j_3j_4j_5}\leq
%\frac 1{|j_1|^{m+1/2-j}|j_2|^{m+1/2-k} |j_3|^{3/2}
%|j_4|^{3/2} |j_3|^{1/2}}
%\end{equation}
Following the same argument used to prove \eqref{A5678} it is sufficient 
to prove 
 \eqref{A5678I} with the condition $\vec j\in {\mathcal A}_5$
replaced by $\vec j\in \tilde {\mathcal A}_5$ and
$\vec j\in \tilde {\mathcal A}_5^c$ (where we used the notation $\vec j=(j_1, j_2, j_3, j_4, j_5)$).
In the first case  (i.e. when we consider the constraint $\vec j\in \tilde {\mathcal A}_5$
in \eqref{A5678I})
we can combine Corollary \ref{5,5}
with an orthogonality argument 
and by \eqref{conSTRI},  and \eqref{A5678I} follows by
\begin{equation}\label{estimatepro5}\sum_{\substack{0<|j_1|, 
|j_2|, |j_3|, |j_4|, |j_5|\leq N\\
|j_1+j_2|>N\\(j_1,j_2,j_3,j_4,j_5)\in {\mathcal A}_5}} \frac 1{|j_1|^{3}
|j_2| |j_3|^{3} |j_4|^{3}|j_5|}
=O\Big( \frac{\ln^2N}N\Big)\end{equation}
%where $j,k$ satisfy \eqref{conne5}.
where we used Lemma \ref{algebrTV}.
Next we focus on the proof of \eqref{A5678I} in the case that 
the constraint $\vec j\in {\mathcal A}_5$ is replaced by 
$\vec j\in \tilde {\mathcal A}_5^c$. Then arguing as above 
it is sufficient to prove
\begin{equation}\label{A5678909998tzvI}
\sum_{j=1}^N \Big \| \sum_{\substack
{0<|j_1|, 
|j_2|, |j_3|, |j_4|, |j_5|\leq N\\
|j_1+j_2|>N\\(j_1,j_2,j_3,j_4,j_5)\in {\mathcal B}_5^{c,j, (l,q)}}}
b_{j_1j_2j_3j_4j_5}g_{j_1}g_{j_2}g_{j_3}g_{j_4}g_{j_5}  \Big\|_{L^2(dp)}
\rightarrow 0 \end{equation}
with 
$(l,q)= (1,3), (1,4), (1,5), (2,3), (2,4), (2,5)$.
In fact by looking at the symmetries of the r.h.s. in \eqref{AbCf}
we can restrict to $(l,q)=(1,3), (1,5), (2,3), (2,5)$.
\begin{itemize}
\item $(l,q)=(1,5)$. By using the orthogonality stated in Corollary \ref{orthtzv}, 
and by recalling \eqref{AbCf}, we deduce that the estimate \eqref{A5678909998tzvI}
follows provided that : 
\begin{equation}\label{A5678909998tzvlm12rsI}
\sum_{j_1=1}^N \Big (\sum_{\substack
{0<|j_2|, |j_3|, |j_4|\leq N\\
|j_3+j_4-j_1|>N
}} \frac 1{|j_1|^{4}|j_2| |j_3|^{3}
|j_4|^{3}}
\Big)^\frac 12
\rightarrow 0 \hbox{ as } N\rightarrow \infty .\end{equation}
%where we have used \eqref{AbCf}.
Using that $\max(|j_1|,|j_3|,|j_4|)>N/3$, we obtain that 
the last expression is
$
O\Big( \frac{\ln^{1/2}(N)}{N}\Big).
$
%In fact this estimate follows by: 
%$$\sum_{j_1=1}^N \frac 1{j_1}\Big (\sum_{\substack
%{0<|j_2|, |j_3|, |j_4|\leq N\\
%|j_1+j_2|>N
%}} \frac 1{|j_1|^{2}|j_2| |j_3|^{3}
%|j_4|^{3}}
%\Big)^\frac 12= O\Big( \frac{\ln^{3/2} N}{\sqrt N}\Big)$$
%where we used that $\max(|j_1|,|j_3|,|j_4|)>N/3$.
%%%%%
\item $(l,q)=(1,3)$. Arguing as above we are reduced to prove
\begin{equation*}
\sum_{j_1=1}^N \Big (\sum_{\substack
{0<|j_2|, |j_3|, |j_4|\leq N\\
|j_1+j_2|>N
}} \frac 1{|j_1|^{6}|j_2| |j_4|^{3}
|j_5|}
\Big)^\frac 12
\rightarrow 0 \hbox{ as } N\rightarrow \infty \end{equation*}
which follows by 
$$\sum_{j_1=1}^N \frac 1{j_1^2}\Big (\sum_{\substack
{0<|j_2|, |j_4|, |j_5|\leq N\\
|j_1+j_2|>N
}} \frac 1{|j_1|^{2}|j_2| |j_4|^{3}
|j_5|}
\Big)^\frac 12=O\Big (\frac{\ln N}{\sqrt N}\Big)$$
where we used Lemma \ref{algebrTV} (observe that we introduced artificially a sum in $j_1$). 
%%%%
\item $(l,q)=(2,5)$. Arguing as above we are reduced to prove
\begin{equation*}
\sum_{j_2=1}^N \Big (\sum_{\substack
{0<|j_1|, |j_3|, |j_4|\leq N\\
|j_1+j_2|>N
}} \frac 1{|j_1|^{3}|j_2|^2 |j_3|^{3}
|j_4|^3}
\Big)^\frac 12
\rightarrow 0 \hbox{ as } N\rightarrow \infty \end{equation*}
which follows by 
\begin{equation*}
\sum_{j_2=1}^N \Big (\sum_{\substack
{0<|j_1|\leq N\\
|j_1+j_2|>N
}} \frac 1{|j_1|^{3}|j_2|^2}
\Big)^\frac 12= O\Big( \frac {1}{\sqrt N}\Big)
\end{equation*}
where we have used Lemma \ref{serienew}.
%%%%%%
\item $(l,q)=(2,3)$. Arguing as above we are reduced to prove
\begin{equation*}
\sum_{j_2=1}^N \Big (\sum_{\substack
{0<|j_1|, |j_4|, |j_5|\leq N\\
|j_1+j_2|>N
}} \frac 1{|j_1|^{3}|j_2|^4 |j_4|^{3}
|j_5|}
\Big)^\frac 12
\rightarrow 0 \hbox{ as } N\rightarrow \infty \end{equation*}
which follows by 
\begin{equation*}
\sum_{j_2=1}^N \frac 1{j_2}\Big (\sum_{\substack
{0<|j_1|, |j_4|, |j_5|\leq N\\
|j_1+j_2|>N
}} \frac 1{|j_1|^{3}|j_2|^2 |j_4|^{3}
|j_5|}
\Big)^\frac 12=O\Big( \frac{\ln ^{3/2}N}{ N}\Big)
\end{equation*}
where we used the fact that $|j_1+j_2|>N$ implies $\max\{j_1, j_2\}> N/2$.
\end{itemize}
This completes the proof of Lemma~\ref{odd2}.
\end{proof}
%%%%%%%%%%%%%%%%%%%%%%%%%%%
\begin{lem}\label{odd25} For every $k\geq 5$, $m\geq 1$ we have:
\begin{equation*}
\lim_{N\rightarrow \infty} \Big \|\int p^*_N(\pi_N u) dx\Big \|_{L^2(d\mu_{m+1/2})}=0
\end{equation*}
where 
\begin{equation}\label{hyps}
p(u)\in {\mathcal P}_k(u), |p(u)|\leq m-1, \|p(u)\|\leq 2m-2.
\end{equation}
\end{lem}
%%%%%%%%%%%%
\begin{proof}
Again, we will consider only some specific $p(u)$ that satisfy the assumptions \eqref{hyps}.
Arguing as at the end of Proposition~\ref{necv2}, one deduce that
the argument we give is rather general and can be adapted to any $p(u)$ that satisfies \eqref{hyps}.
More precisely we assume:
\begin{equation}\label{basicstruct456fin}
p(u)=\Pi_{j=1}^k \partial_x^{\alpha_j} u,\end{equation}
\begin{equation}\label{constrained} \sum_{j=1}^k \alpha_j\leq 2m-2
\hbox{ and } \sup_{j=1,...,k} \{\alpha_j\}\leq m-1.\end{equation}
Then we have the following expression 
$$\int p_N^*(u) dx= \sum_{i=1}^k \int \big(\Pi_{\substack{j=1\\j\neq i}}^k \partial_x^{\alpha_j} u \big)
\partial_x^{\alpha_i} \pi_{>N} (u\partial_x u)
dx=\sum_{i=1}^kI_i(u).$$
Hence the proof follows by $\|I_i(\pi_Nu)\|_{L^2(d\mu_{m+1/2})}\rightarrow 0$ 
as $N \rightarrow \infty$ for every $i=1,...,k$.

By symmetry we focus on $I_k(\pi_N u)$, all the other terms are equivalent.
By the Leibnitz rule, it is sufficient to estimate expressions of the following type
\begin{equation}\label{F5N}\|F_N(\pi_N u)\|_{L^2(d\mu_{m+1/2})}\rightarrow 0
\hbox{ as } N\rightarrow \infty\end{equation}
where $F_N(u)= \int \big(\Pi_{j=1}^{k-1} \partial_x^{\alpha_j} u \big) 
\pi_{>N} (\partial_x^\alpha u \partial_x^\beta u) dx
$ with 
\begin{equation}\label{constfin}\alpha<m, \beta \leq m, \alpha_j\leq m-1, j=1,..., k-1.
\end{equation}
In fact the estimate \eqref{F5N} follows by
\begin{equation}\label{A5678fin}\Big \|\sum_{\substack{0<|j_1|,..., 
|j_{k+1}|\leq N\\
|j_k+j_{k+1}|>N\\(j_1,...,j_{k+1})\in {\mathcal A}_{k+1}}} 
b_{j_1...j_{k+1}}
g_{j_1}...g_{j_{k+1}}
\Big \|_{L^2(dp)}\rightarrow 0\end{equation}
where
$$b_{j_1...j_{k+1}}=\frac 1{\Pi_{i=1}^{k-1} |j_i|^{m+1/2-\alpha_i}
|j_k|^{m+1/2-\alpha}|j_{k+1}|^{m+1/2-\beta}}.$$
By \eqref{constfin} we get
$$|b_{j_1...j_{k+1}}|\leq \frac 1{\Pi_{i=1}^{k-1} |j_i|^{3/2}
|j_k|^{3/2}|j_{k+1}|^{1/2}}$$
and hence by the Minkowski inequality
\eqref{A5678fin} follows by 
$$\sum_{\substack{0<|j_1|,..., 
|j_{k+1}|\leq N\\
|j_k+j_{k+1}|>N\\(j_1,...,j_{k+1})\in {\mathcal A}_{k+1}}}  
\frac 1{\Pi_{i=1}^{k-1} |j_i|^{3/2}
|j_k|^{3/2}|j_{k+1}|^{1/2}}
$$
$$
\leq \sum_{\substack{0<|j_1|,..., 
|j_{k}|\leq N\\
|j_1+...+ j_{k-1}|>N}}  
\frac 1{\Pi_{i=1}^{k-1} |j_i|^{3/2}
|j_k|^{3/2}}
$$$$
\leq C \sum_{\substack{0<|j_1|,..., 
|j_{k-1}|\leq N\\
|j_1+...+ j_{k-1}|>N}}  
\frac 1{\Pi_{i=1}^{k-1} |j_i|^{3/2}}
=O\Big( \frac 1{\sqrt N}\Big)$$
where we have used the fact that
if $|j_1+...+ j_{k-1}|>N$ then $\max \{|j_1|,...,|j_{k-1}|\} \geq \frac N{k-1}$.
This completes the proof of Lemma~\ref{odd25}.
\end{proof}
%%%%%%%%%%%%%%%%%%%%%%%%%%%%%%%%%%%%%%%%%%%%%%%%%%%%%%
\section{Proof of Theorem \ref{prcv} for $k=2m+1$}\label{odd3D}
We shall need some extra informations (compared with the
description given in \eqref{odd}) on the structure
of the conservation laws $E_{k/2}$ with $k$ odd.
\begin{prop}\label{prop:cubstruc}
Let $k=2m+1$. Then one may assume that the only terms of the second term in the right hand-side of \eqref{odd} are given by (up to coefficient):
$$
\int u (H \partial_x^m u) (H\partial_x^{m}u) dx,
\int u \partial_x^m u \partial_x^{m}u dx,  \int u (H \partial_x^m u) \partial_x^{m}u dx
$$
for a suitable constant $c$.
\end{prop}
\begin{proof}
In \cite{TV2} we have introduced the
following  notation. We note by $\Lambda^0$ the identity map while for $n\geq 1$, 
$\Lambda^n$ stays for an operator of the form $(c_1+c_2H)\partial_x^m$, where $c_1$, $c_2$ are constants.
Then we have proved in \cite{TV2} (see Lemma 2.4), following Matsuno \cite{Mats}, that the cubic contribution on the second term in the r.h.s. in \eqref{odd}
is given by
\begin{equation}\label{zvezda}
\int \Lambda^{j_1}(u)\Lambda^{j_2}(u)\Lambda^{j_3}(u) dx
\end{equation}
$$j_1+j_2+j_3=2m.$$
The conclusion follows as in \cite{TV2}.
This completes the proof of Proposition~\ref{prop:cubstruc}.
\end{proof}
\begin{remark}
The main interest in Proposition \ref{prop:cubstruc}
is that in order to prove the property
$\lim_{N\to \infty} \|G_N^{m+1/2}(u)\|_{L^2(d\mu_{m+1/2})}=0$ 
(which by Proposition 3.4 in \cite{TV2} 
is related with $\lim_{N\rightarrow \infty}\|\int p_N^*(\pi_N)u dx\|_{L^2(d\mu_{m+1/2})}=0$), then we can exclude from our analysis the terms $p(u)$
than in principle could appear in the second term of the right hand-side of \eqref{odd} 
and which are different from the ones which are isolated in Proposition
\ref{prop:cubstruc}.
In fact this is more than a simplification, since it is easy to check that there exist
$p(u)$ which in principle could appear in the second term of the right hand-side of \eqref{odd} and for which our analysis in Lemma~\ref{odd1} fails.
\end{remark}
%%%%%%
Next, we introduce for every $k=0,\dots, 2m+1$ the function
$$G_N^{k/2}(u_0): {\rm supp} ( \mu_{m+1/2}) \ni u_0 \longmapsto \frac d{dt} E_{k/2}(\pi_{N}\Phi_N^t(u_0))_{t=0}.$$
According to Proposition~5.4 in \cite{TV2}, in order to prove Theorem~\ref{prcv} for $k=2m+1$ it is sufficient to prove
the following proposition.
\begin{prop}\label{necvodd}
Let $m\geq 1$, then we have
\begin{equation*}
\lim_{N\to \infty} \sum_{k=0}^{2m+1} \|G_N^{k/2}(u)\|_{L^2(d\mu_{m+1/2})}=0.
\end{equation*}
\end{prop}
\begin{remark}\label{k=2remoddi}
The property 
$\lim_{N\to \infty} \sum_{k=0}^1 \|G_N^{k/2}(u)\|_{L^2(d\mu_{m+1/2})}=0$
follows from the fact that $E_{1/2}(\pi_N \Phi_t^N(u))\equiv const$ and $\|\pi_N \Phi_t^N(u)\|_{L^2}\equiv const$. Recall also that by Proposition~\ref{necv}
we get $\lim_{N\to \infty}\|G_N^{1}(u)\|_{L^2(d\mu_1)}
=0$, that in turn implies
$\lim_{N\to \infty}\|G_N^{1}(u)\|_{L^2(d\mu_{3/2})}
=0$ (for a proof of this last fact
see the end of the proof of Proposition \ref{necv2}
where the property $\|G_N^{3/2}(u)\|_{L^2(d\mu_2)}\rightarrow 0$ is recovered
from the knowledge $\|G_N^{3/2}(u)\|_{L^2(d\mu_{3/2})}\rightarrow 0$).
As a consequence of those facts we deduce Proposition~\ref{necvodd} 
for $m=1$ provided that  
$\lim_{N\to \infty}\|G_N^{3/2}(u)\|_{L^2(d\mu_{3/2})}
=0$.
\end{remark}
\begin{remark}\label{k=2remodd}
We claim that, for any $m\geq 1$, Proposition~\ref{necvodd}
follows provided that we prove
\begin{equation}\label{nom}\lim_{N\to \infty}\|G_N^{n+1/2}(u)\|_{L^2(d\mu_{n+1/2})}
=0,\quad \forall\,\, n\leq m. 
\end{equation}
If we have \eqref{nom} then one can deduce that
$\lim_{N\to \infty}\|G_N^{n+1/2}(u)\|_{L^2(d\mu_{m+1/2})}
=0$ (for a proof of this fact see Remark~\ref{k=2remoddi} and the second part
of the proof of Proposition~\ref{necv2}).
Moreover, we also have that 
$\|G_N^{l}(u)\|_{L^2(d\mu_{l})}
\rightarrow 0$ as $N\rightarrow \infty$ 
for every integer $l<m+1/2$. Indeed, for $l\geq 3$ this property is proved in  \cite{TV2} and for 
 $l=1,2$ it follows from Propositions~\ref{necv} and \ref{necv2}.
Arguing as above, it implies $\|G_N^{l}(u)\|_{L^2(d\mu_{m+1/2})}
\rightarrow 0$ as $N\rightarrow \infty$, for every integer $l<m+1/2$.
Thus indeed, 
 Proposition~\ref{necvodd}
follows provided that we prove \eqref{nom}.
\end{remark}
\begin{proof}[Proof of Proposition \ref{necvodd}]
According to Remark~\ref{k=2remodd} it is sufficient to prove 
\eqref{nom}, i.e. 
$\lim_{N\to \infty}\|G_N^{m+1/2}(u)\|_{L^2(d\mu_{m+1/2})}=0$
for any $m\geq 1$. 
This property follows 
by combining the explicit expression of 
$G_N^{m+1/2}(u_0)$
(see Proposition~3.4 in \cite{TV2})
with Lemma \ref{odd1}, \ref{odd3}, \ref{odd2}, \ref{odd25} and
Proposition \ref{prop:cubstruc} (where are isolated the unique cubic 
terms that are involved in the second term of the right hand-side of \eqref{odd}
and that are treated in Lemma~\ref{odd1}). 
%hen we get
%$\|G_N^{m+1/2}\|_{L^2(d\mu_{m+1/2})}\rightarrow 0$.
%The conclusion follows arguing as in \cite{TV2}.
\end{proof}
%%%%%%%%%%%%%%%%%%%%%%%%%%%
\section{Proof of Theorem \ref{k>5}}
Having Theorem~\ref{prcv} at our disposal,
the proof of Theorem~\ref{k>5} for $k\geq 5$ is the same as in our previous work \cite{TV2}.
Indeed, as in \cite{TV2}, the proof of  \eqref{finitecauchy} follows by the classical energy method.
Therefore, in this section, we shall only treat the case $k=4$ which, compared to \cite{TV2}, 
needs some modifications using the dispersive effect in the local analysis. 

Along this section we denote by $B^\sigma(R)$ the ball of radius $R$, centered at the origin of $H^\sigma$.
\begin{prop}\label{boNno}
There exists $\gamma>0$ such that for every $s >5/4$ there exist $C>0$ and $c_s>0$ 
such that for every $R\geq 1$, if we set $T=c_s R^{-\gamma}$ then for every $u_0\in B^{s}(R)$, every $N\geq 1$,
$$
\|\Phi_t(u_0)\|_{L^\infty([0,T];H^{s})}\leq R+R^{-1},\quad
\|\partial_x \Phi_t(u_0)\|_{L^1([0,T];L^{\infty})}\leq CR^2
$$
and
$$
\|\Phi^N_t(u_0)\|_{L^\infty([0,T];H^{s})}\leq R+R^{-1}\,,\quad
\|\partial_x \Phi^N_t(u_0)\|_{L^1([0,T];L^{\infty})}\leq CR^2\,.
$$
% and $u_0\in B^{5/4+\epsilon}(R)$. Then there exist
%$T, C=T(\epsilon, R), C(\epsilon, R)>0$ such that
%$$\max\{\|\partial_x u\|_{L^1(0,T) L^\infty_x}, \|u\|_{L^\infty(0,T)H_x^{5/4+\epsilon}}\}<C$$
 %where $u$ solves \eqref{bo}. 
 \end{prop}
 \begin{proof}[Proof of Proposition~\ref{boNno}]
 We only prove the first a priori bound, the proof of the second being very similar since the projector $\pi_N$ does not affect the analysis we perform below.
We shall need the following version of Strichartz estimates
for $e^{tH\partial_x^2}$. Its proof follows by the same estimate on 
operator $e^{it\partial_x^2}$ in the periodic case,
first proved in \cite{BGT} in the setting of a general manifold.
\begin{lem}\label{strloc}
There exists $C>0$ such that for every $L=2^l$, $l\in \N$ we have
$$\|e^{tH\partial_x^2} (\Delta_L u_0)\|_{L^4([0, 1/L];L^\infty)}
\leq C \|\Delta_L u_0\|_{L^2} $$ and 
$$
\|\int_0^t e^{(t-s)H\partial_x^2} \Delta_L G(s) ds\|_{L^4([0, 1/L];L^\infty)}
\leq C \|\Delta_L G\|_{L^1([0, 1/L];L^2)}
$$
where $\{\Delta_L\}_{L=2^l, l\in \N}$ is the usual Littlewood-Paley partition.
\end{lem}
%%%%%%%%%%%
Here is the key a priori estimate.
\begin{lem}\label{lemboNno}
Let $s>5/4$. Then there exists $C$ such that for every $T\in [0,1]$,
\begin{equation*}
%\label{lemmac}
\|u\|_{L^1([0,T];L^\infty)}+\|\partial_x u\|_{L^1([0,T];L^\infty)}\leq 
C T^{3/4}\big(\|u\|_{L^\infty([0,T];H^{s})} + \|u\|_{L^\infty([0,T];H^{s})}^2\big)
\end{equation*}
where $u$ is a smooth solution to \eqref{bo}.
\end{lem}
%%%%
\begin{proof} Let $\Delta_L$ be a Littlewood-Paley partition and assume that
\begin{equation}\label{BOF}v_t+H\partial_x^2 v=\partial_x F
\end{equation}
For every interval $I=[a, b)$ with $|a-b|=T/L$ we get by the H\"older inequqlity
\begin{equation}\label{koch0}\|\Delta_L v\|_{L^1_I L^\infty_x}\leq |I|^{3/4} \|\Delta_L u\|_{L^4_IL^\infty_x} 
\end{equation}
and hence, by the integral formulation of \eqref{BOF} in conjunction with Lemma~\ref{strloc}, we can continue
\begin{multline}\label{mumar}
...\leq C T^{3/4} L^{-3/4} \|e^{tH\partial_x^2} \Delta_L v(a)\|_{L^4_IL^\infty_x} + C T^{3/4}
L^{-3/4}\|\Delta_L \partial_x F\|_{L^1_I L^2_x}
\\
\leq C T^{3/4}L^{-3/4} \|\Delta_L v(a)\|_{L^2_x} + C T^{3/4}
\|\Delta_L F\|_{L^1_I H^{1/4}_x} \,.
\end{multline}
Next we split the interval $[0,T]$ as a disjoint union
$[0, T]=\cup_{j=1}^L I_j$ with $|I_j|=T/L$ and
$I_j=[a_j, a_{j+1}]$. Then by combining the estimates \eqref{koch0} and
\eqref{mumar} we get
\begin{equation}\label{Ij}
\|\Delta_L v\|_{L^1_{I_j} L^\infty_x}\leq C T^{3/4}L^{-3/4} 
\|\Delta_L v(a_j)\|_{L^2_x} + C T^{3/4}
\|\Delta_L F\|_{L^1_{I_j} H^{1/4}_x} 
\end{equation}
for every $ j=1,...,L$.
Next we consider the sum for $j=1,...,L$ of the estimates \eqref{Ij}. First notice that
\begin{multline}\label{koch2}
\sum_{j=1}^L L^{-3/4} \|\Delta_L v(a_j) \|_{L^2}
=  \sum_{j=1}^L L^{1/4+\epsilon} \|\Delta_L v(a_j) \|_{L^2} L^{-1-\epsilon}
\\
\leq \sum_{j=1}^L \|v(a_j) \|_{H^{1/4+\epsilon}} L^{-1-\epsilon}
\leq C L^{-\epsilon} \|v\|_{L^\infty([0,T]; H^{1/4+\epsilon})}\,.
\end{multline}
On the other hand we have
\begin{equation}\label{koch3}
 \sum_{j=1}^L \|\Delta_L v\|_{L^1_{I_j} L^\infty}= \|\Delta_L v\|_{L^1([0,T]; L^\infty_x)}
\end{equation}
and
\begin{equation}\label{koch4}
 \sum_{j=1}^L \|\Delta_L F\|_{L^1_{I_j} H^{1/4}}= 
 \|\Delta_L F\|_{L^1([0,T] ;H^{1/4})}\,.
\end{equation}
By combining \eqref{Ij}, \eqref{koch2}, \eqref{koch3} and \eqref{koch4} then 
we get 
\begin{equation}\label{finkoch}\|\Delta_L v\|_{L^1([0,T]; L^\infty)}\leq C 
T^{3/4}
\big(L^{-\epsilon} \|v\|_{L^\infty([0,T]; H^{1/4+\epsilon})}
+
\|\Delta_L F\|_{L^1([0,T]; H^{1/4})} 
\big).
\end{equation}
Next, notice that
$$\sum_L \|\Delta_L F\|_{L^1([0,T]; H^{1/4})}
\leq \sum_L \|\Delta_L F\|_{L^1([0,T]; H^{1/4+\epsilon})} L^{-\epsilon}
\leq C\|F\|_{L^1([0,T]; H^{1/4+\epsilon})} .
$$
Hence by summing in $L$ the estimates \eqref{finkoch},
and by using the Minkowski inequality on the l.h.s., we deduce
\begin{equation}\label{quasko}
\|v\|_{L^1([0,T];L^\infty)}\leq C T^{3/4}\|v\|_{L^\infty([0,T]; H^{1/4+\epsilon})} + C
T^{3/4}\|F\|_{L^1([0,T]; H^{1/4+\epsilon})}
\end{equation}
where $v,F$ are related by \eqref{BOF}.
Next, notice that if we denote $v=\partial_x u$ then by \eqref{bo}
we get
$$
\partial_t v+ H \partial_x^2 v= \partial_x (u v).
$$ 
Hence we can apply the estimate \eqref{quasko} (with $F=uv$)
and we get
\begin{equation}\label{quaskov}
\|v\|_{L^1([0,T];L^\infty)}\leq C T^{3/4}\|v\|_{L^\infty([0,T]; H^{1/4+\epsilon})} + CT^{3/4}
\|uv\|_{L^1([0,T]; H^{1/4+\epsilon})}
\end{equation}
which implies, by recalling  that $v=\partial_x u$,
\begin{equation}\label{quaskovtr}
\|\partial_x u\|_{L^1([0,T];L^\infty)}\leq C T^{3/4}\|u\|_{L^\infty([0,T]; H^{5/4+\epsilon})} + C
T^{3/4}\|u^2\|_{L^1([0,T]; H^{5/4+\epsilon})}.
\end{equation}
%$$\leq C T^{3/4}\|u\|_{L^\infty_{(0,T)} H^{5/4+\epsilon}} + C
%T^{3/4} \|u\|_{L^\infty_{(0,T)} L^{\infty}_x}\|u\|_{L^\infty_{(0,T)} H^{5/4+\epsilon}_x}$$
We conclude since $H^{5/4+\epsilon}$ is an algebra.
The estimate for $\|u\|_{L^1([0,T];L^\infty)}$ is very similar (simpler) and hence its proof will be omitted.  
This completes the proof of Lemma~\ref{lemboNno}.
\end{proof}
%%%%%%%%%
Let us now complete the proof of Proposition~\ref{boNno}.
The classical energy inequality and the Kato-Ponce commutator estimate 
(see \cite{Ponce}) yield
\begin{equation}
\|u\|_{L^\infty([0, t];H^{s})} \leq \|u(0)\|_{H^{s}}\exp\big(C(\|\partial_x u\|_{L^1([0,t]; L^\infty)}
+\| u\|_{L^1([0,t]; L^\infty)})
\big)
\end{equation}
(the term $\| u\|_{L^1(0,t) L^\infty_x}$ appears when localizing the estimate in the euclidean case to the case of the torus).
Hence by Lemma~\ref{lemboNno}, we get:
$$
X(t)\leq \|u(0)\|_{H^{s}} e^{CT^{3/4}(X(t) + (X(t))^2)}\, ,\quad  \forall t\in [0, T],
$$
where
$$
X(t)= \|u\|_{L^\infty([0, t];H^{s})}
$$ 
and $T>0$ is to be fixed as in the statement of Proposition~\ref{boNno}.
This implies that 
\begin{equation}\label{impgal}
\{X(t),\,\,t\in (0, T)\}\subset 
\{y\geq 0\,:\,F_{T}(y)\leq R\}\,,
\end{equation}
where $F_T(y)=ye^{-T^{3/4}(y + y^2)}$.
Observe that 
$F_T(0)=F_T(\infty)=0$ and there exists $y_{T}\sim T^{-3/4}$ such that $y_T$ is the maximum of  $F_T$ on $[0,+\infty)$.
Moreover, $F_T$ is increasing in $[0,y_T]$ and decreasing in $[y_T,\infty)$.
We choose $T=cR^{-\gamma}$, where $c\ll 1$ is a small constant. 
We observe that, if $\gamma>3/4$ then 
$R<y_{T}$.
Moreover, a direct computation shows that for $\gamma$ large enough and $c$ small enough, one has
$
F_{T}(R+R^{-1})>R.
$
Therefore by a continuity argument 
$X(t)\leq R+R^{-1}$.
%By looking at the function $F_T$ we deduce that
%\begin{equation}\label{TiNes}F_T(0)=0=F_T(\infty) \hbox{ and }  \lim_{T\rightarrow 0} \max_{\R^+} F_T=\infty\end{equation}
%Moreover, if we denote by $y_T\in \R^+$ the unique point such that
%$F_T(y_T)= \max_{\R^+} F_T$ then we get:
%\begin{equation}\label{tines}\lim_{T\rightarrow 0} y_T=\infty\end{equation}
%Notice also that $F_T$ is increasing on $(0, y_T)$ and decreasing on $(y_T, \infty)$.
%A first consequence of the informations above is the following:
%\begin{equation}\label{contop}\forall T>0, \forall 0<\lambda<\max_{\R^+} F_T,
%\exists 0<x_{\lambda, T}< y_T< z_{\lambda, T} \hbox{ s.t. }\end{equation}
%$$
%\{y>0 | F_T(y)\leq \lambda\}=(0, x_{\lambda, T}]\cup [z_{\lambda,T}, \infty)$$
%Thanks to \eqref{TiNes} and \eqref{tines}
%we can choose $T=T(R)>0$ (small enough) such that 
%\begin{equation}\label{yT}
%y_T>R \hbox{ and }F_T(y_T)=\max_{\R^+} F_T> R\end{equation}
%By using \eqref{contop} where we fix $\lambda=R$ (and $T$ is chosen as above), then
%we get
%the existence of $x_R,z_R\in \R^+$ such that
%\begin{equation}\label{tzvserg}
%\{y>0 | F_T(y)\leq R\}=(0, x_{R}]\cup [z_{R}, \infty), 0<x_R<y_T<z_R
%\end{equation}
%Notice also that $X(0)=\|u_0\|_{H^{5/4+\epsilon}_x}\leq R <y_T$ (where we used \eqref{yT})
%and hence by combining \eqref{impgal} with \eqref{tzvserg} 
%we get $X(0)\in (0, x_R]$.
%By combining this fact with the continuity of the function $t\rightarrow X(t)$
%and by using again \eqref{impgal} and \eqref{tzvserg}, we get
%$X(t)\in (0, x_R)$ for every $t\leq T$.
This completes the proof of Proposition~\ref{boNno}.
\end{proof}

%The proof of Proposition \ref{boNno} works also
%in the case of the truncated flow, with uniform constants in $N$. Hence we get 
%\begin{prop}\label{boNnoN}
%Let $\epsilon,R>0$ be fixed and $u_0\in B^{5/4+\epsilon}(R)$. Then there exist
%$T, C=T(R), C(R)>0$ such that
%$$\sup_N \max\{\|\partial_x \pi_N w_N\|_{L^1(0,T) L^\infty_x}, \|\pi_N w_N\|_{L^\infty(0,T)H^{5/4+\epsilon}_x}\}<C$$
% where $w_N(t,x)$ solves \eqref{BONN}. 
 %\end{prop}
The next proposition is a version of \eqref{finitecauchy} in the case $k=2$.
\begin{prop}\label{kochprop}
Fix $5/4<s<\sigma<3/2$.
For every $R>0$ we have
\begin{equation}\label{eq:claimjkl}
\lim_{N\rightarrow \infty} \Big (\sup_{\substack{t\in [0,T]\\u_0\in B^{\sigma}(R)}} 
\|\Phi_t(u_0)- \Phi^N_t(u_0)\|_{H^{s}}\Big )=0,
\end{equation}
where $T=c_{\sigma}R^{-\gamma}$ is fixed in Proposition~\ref{boNno}.
\end{prop}
Thanks to Proposition~\ref{boNno}, the proof of Proposition~\ref{kochprop} is a straightforward adaptation of the proof of \cite[Proposition~4.1]{TV2}.

Thanks to Proposition~\ref{kochprop}, as in \cite[Lemma~5.6]{TV2} one can show that for every compact set $K$ of $H^\sigma$  one has $\rho_{2,R}(K)\leq \rho_{2,R}(\Phi_t(K))$, where $t$ is sufficiently small, depending on $K$. In order to extend this property to any time one needs the following statement.
\begin{prop}\label{soft}
Let $\sigma>5/4$, $t>0$ and $K$ a compact set of $H^\sigma$. Then there exists $R>0$ such that 
$
\{
\Phi_{\tau}(K),\,0\leq \tau\leq t
\}\subset B^{\sigma}(R).
$
\end{prop}
Proposition~\ref{soft} follows from the propagation of higher Sobolev regularity which is a byproduct of the well-posedness result of Molinet \cite{M}.
\begin{remark}
Observe that one may prove a suitable substitute of Proposition~\ref{soft} which avoids the use of \cite{M}.
More precisely thanks to Proposition~\ref{boNno} (the precise form $R+R^{-1}$ of the bound is of importance) , we can prove the analogue of \cite[Proposition~8.6]{BTT} which in turn implies the suitable substitute of
Proposition~\ref{soft}. 
The advantage of such an argument is that it does not rely on on a global
regularity result on the support of the measure and it uses the measure invariance to get long time regularity bounds.  
\end{remark}
Using Proposition~\ref{soft}, one can complete the proof of Theorem~\ref{k>5} in the case $k=4$ exactly as in the last page of \cite{TV2}.
\section{Appendix.  A brief comparing of the recurrence properties of the Benjamin-Ono and KdV equations flows}
%%%%%%%%%%%%%%%%%%%%%%%%%
\subsection{Recurrence properties of the KdV equation}
Consider the KdV equation, posed on the torus
\begin{equation}\label{KdV}
\partial_t u+\partial_x^3 u+u\partial_x u=0,\quad (t, x)\in\R\times {\R}/{2\pi \Z}
\end{equation}
with initial data 
\begin{equation}\label{KdV-data}
u(0)=u_0\in H^s(\R/2\pi \Z;\R),\quad s=0,1,2,\dots 
\end{equation}
The KdV equation is another fundamental dispersive model which is better understood compared to the Benjamin-Ono equation.
The problem \eqref{KdV}-\eqref{KdV-data} is globally well-posed (see \cite{B0}).
In particular the solution may be seen as a continuous curve in $C(\R;H^s)$.
This well-posendess result says little about the long time behavior of the solutions.
One can however prove the following remarkable statement concerning the KdV flow (see \cite{B0, MT}).
\begin{thm}\label{McKean}
The KdV flow is  {\bf almost periodic} in time. Namely, if $u$ is a solution of \eqref{KdV}-\eqref{KdV-data} then
for every $\varepsilon$ there exists an almost period $l_{\varepsilon}$ such that for every interval $I$ of size 
$\geq l_{\varepsilon}$ there exists $\tau\in I$ such that for every $t\in\R$,
$$
\|u(t+\tau)-u(t)\|_{H^s}<\varepsilon.
 $$
\end{thm}
The proof of this result is based on the solution of the inverse spectral problem associated to the Hill operator
$-\partial_x^2+V(x)$, where $V\in H^s$ is a periodic potential. 

A direct consequence of Theorem~\ref{McKean} is the following statement.
%%%
\begin{cor}\label{c-c}
The KdV flow is {\bf recurrent} in time  : for {\bf every} $u_0\in H^s$ there is a sequence $(t_n)$ going to infinity such that the corresponding solution of
\eqref{KdV}-\eqref{KdV-data} satisfies
$$
\lim_{n\rightarrow \infty}\|u(t_n)-u_0\|_{H^s}=0.
$$
\end{cor}
%%%%%%%%%%%%%%%%%%%%%%%%%%%%%%%%%%%%%%%%%%%%%%%%%%%%
\subsection{Recurrence properties of the Benjamin-Ono equation}
Consider now the Benjamin-Ono (BO) equation, posed on the torus
\begin{equation}\label{Bo}
\partial_t u+H\partial_x^2 u+u\partial_x u=0,\quad (t,x)\in\R\times {\R}/{2\pi \Z} 
\end{equation}
with initial data 
\begin{equation}\label{bo-data}
u(0)=u_0\in H^s(\R/2\pi \Z;\R),\quad s=0,1,2,\dots 
\end{equation}
Recall that the problem \eqref{Bo}-\eqref{bo-data} is globally well-posed \cite{M}.

Consider the initial data for \eqref{Bo} of the form
\begin{equation}\label{1}
u_{0}(x)=\sum_{n\in \Z^{\star}}\frac{g_{n}(\omega)}{|n|^{k/2}}e^{inx},\quad k=4,5,6,7, \dots
\end{equation}
where $g_n(\omega)=h_n(\omega)+il_n(\omega)$, $h_n,l_n\in {\mathcal N}(0,1)$,
$(h_n,l_n)_{n>0}$ are independent and $g_{-n}=\overline{g_n}$.
Then Theorem~\ref{k>5} and the Poincar\'e recurrence theorem imply the following statement.
\begin{thm}
For almost every $\omega$ the solution of the Benjamin-Ono equation with data given by \eqref{1} is {\bf recurrent} 
(with convergence in $H^s$, $s<\frac{k}{2}-\frac{1}{2}$).
\end{thm}
We also have the following deterministic corollary.
\begin{cor}
Fix an integer $m\geq 0$.
Then there exists a dense set $F_m$ of $H^m(\T)$ such that for every $u_0\in F_m$ the solution of the Benjamin-Ono equation with data $u_0$ is recurrent.  
\end{cor}
In view of the discussion for the KdV equation one may ask the following questions :
\\[0.2cm]
Question 1 : Can we take $F_m=H^m$ ?
\\
Question 2 : Is the Benjamin-Ono flow almost periodic, at least for small data ?
\\

Let us conclude by mentioning that for the Benjamin-Ono equation one may also try to approach the problem of the almost periodicity by the inverse scattering method.
However, one needs to deal with non local operators in sharp contrast with the classical Hill operator occurring in the KdV context. For such operators one may hope to
solve the corresponding direct and inverse problems for small potentials which is the motivation behind the smallness assumption in Question 2.


\begin{thebibliography}{10}

%%%%%%%
\bibitem{ABFS} L. Abdelouhab, J. Bona, M. Felland, J.-C. Saut, {\it Nonlocal models for nonlinear, dispersive waves}, Phys. D 40 (1989) 360-392.
%%%%%%%
\bibitem{B0} J. Bourgain, {\it Fourier transform restriction phenomena for certain lattice subsets and applications to nonlinear evolution equations. II. 
The KdV-equation}, Geom. Funct. Anal. 3 (1993), 209-262
%%%%%%%
\bibitem{B} J. Bourgain, {\it Periodic nonlinear Schr\"odinger equation and invariant measures}, Comm. Math. Phys. 166 (1994) 1-26.
%%%%%%%%%%%%%%%%%%
\bibitem{B2} J. Bourgain, {\it Invariant measures for the 2d-defocusing nonlinear Schr\"odinger equation}, Comm. Math. Phys. 176 (1996) 421-445.
%%%%%%%
\bibitem{BGT} N. Burq, P. G\'erard, N. Tzvetkov, {\it Strichartz inequalities and the nonlinear Schroedinger equation
on compact manifolds}, Amer. J. Math. 126 (2004), 569-605.
%%%%%%%%%%
\bibitem{BP} N. Burq, F. Planchon, {\it On well-posedness for the Benjamin-Ono equation}, Math. Ann. 340 (2008) 497-542.
%%%%
%\bibitem{BT1} N. Burq, N. Tzvetkov, {\it Invariant measure for a three dimensional nonlinear wave
%equation}, Int. Math. Res. Not. IMRN (2007) Art. ID rnm108 26 pp.
%%%%
\bibitem{BT2} N. Burq, N. Tzvetkov, {\it Random data {C}auchy theory for supercritical wave equations. 
{II}. {A} global existence result}, Invent. Math. 173 (2008) 477--496.
%%%%
\bibitem{BTT} N. Burq, L. Thomann, N. Tzvetkov, {\it Long time dynamics for the one dimensional non linear Schr\"odinger equation}, to appear in Ann. Institut Fourier.
%%%%%
\bibitem{D} Y. Deng, Invariance of the Gibbs measure for the Benjamin-Ono equation.
arxiv:1210.1542
%%%%%%%%%%
\bibitem{IK} A. Ionescu, C. Kenig, {\it Global well-posedness of the Benjamin-Ono equation in low regularity spaces}, J. Amer. Math. Soc. 20 (2007) 753-798.
%%%%%%%
%\bibitem{KP} T. Kappeler , J. P\"oschel, {\it KAM and KdV}, Springer, 2003.
%%%%%%%%%%%%%
\bibitem{Ponce} T. Kato, G. Ponce, {\it Commutator estimates and the Euler and Navier-Stokes equations}, Comm. Pure  Appl. Math. 41 (1988) 891-907.
%%%%%%%%
\bibitem{KK} {Kenig, C. and Koenig, K.},
{\it On the local well-posedness of the {B}enjamin-{O}no and
modified {B}enjamin-{O}no equations}, Math. Res. Lett., 10 (2003)
879-895.

\bibitem{KT} H. Koch, N. Tzvetkov, {\it On the local well-posedness of the Benjamin-Ono equation in $H^s(\R)$}, Int. Math. Res. Not. 2003, n.26, 1449-1464.
%%%%%%%%%%%%%%%%%%%%%%%%%%%%%%
%\bibitem{LRS} L. Lebowitz, R. Rose, E. Speer, {\it Statistical dynamics of the Nonlinear Schr\"odinger equation}, J. Stat. Phys. 50 (1988) 657-687.
%%%%%%%
\bibitem{Mats} Y. Matsuno, {\it Bilinear transformation method}, Academic Press, 1984.
%%%%%%
\bibitem{MT}
H.P. McKean, E. Trubowitz, {\it Hill's operator and hyperelliptic function theory in the presence of infinitely many branch points}, 
Comm. Pure Appl. Math. 29 (1976), 143-226.
%%%%%%%%%%
\bibitem{M} L. Molinet, {\it Global well-posendess in $L^2$ for the periodic Benjamin-Ono equation}, Amer. J. Math. 130 (2008) 635-685.
%%%%%%%%%%%%%%%%%%%%%%%%%%%
\bibitem{MP} L. Molinet, D. Pilod,{\it The Cauchy problem of the Benjamin-Ono equation in $L^2$ revisited}, arXiv:1007.1545v1  
%%%%%%
\bibitem{MST} L. Molinet, J.-C. Saut, N. Tzvetkov, {\it  Ill-posedness issues for the Benjamin-Ono and related
equations}, SIAM J. Math. Anal. 33 (2001), 982-988.
%
\bibitem{NORS} A. Nahmod, T.Oh, L. Rey-Bellet, G. Staffilani, 
{\it Invariant weighted Wiener measures and almost sure global well-posedness for the periodic derivative NLS}, J. Eur. Math. Soc. 14 (2012), 1275-1330.
%
\bibitem{OH} T. Oh, {\it Invariance of the {G}ibbs measure for the  {S}chr\"odinger-{B}enjamin-{O}no system}, SIAM J. Math. Anal. 41 (2009/10), 2207-2225.
%
\bibitem{P} G. Ponce, {\it On the global well-posedness of the Benjamin-Ono equation}, 
Diff. Int. Eq. 4 (1991) 527-542.
%
\bibitem{S} A-S. de Suzzoni, {\it Invariant mesure for the cubic non linear wave equation on the unit ball of $R^3$}, Dynamics of PDE 8 (2011), 127-147.
%
\bibitem{T} T. Tao, {\it Global well-posedness of the Benjamin-Ono equation in $H^1$}, J. Hyperbolic Diff. Equations, 1 (2004) 27-49. 
%%%%%%
\bibitem{tz_fourier}
N.~Tzvetkov, {\it Invariant measures for the defocusing NLS}, Ann. Inst. Fourier 58 (2008) 2543-2604.
%%%%%%%%%%%%%%
\bibitem{tz} N. Tzvetkov, {\it Construction of a Gibbs measure associated to the periodic Benjamin-Ono equation}, Probab. Theory Relat. Fields  146 (2010) 481-514.
%%%%%%%%
\bibitem{TV} N. Tzvetkov, N. Visciglia {\it Gaussian measures associated to the higher order conservation 
laws of the  Benjamin-Ono equation }, Ann. Scient. Ec. Norm. Sup. 46 (2013) 249-299.
%%%%%%%%%
\bibitem{TV2} N. Tzvetkov, N. Visciglia {\it Invariant measures and long time behaviour for the  Benjamin-Ono equation}, to appear on Int. Math. Res. Not.
%%%%%%%%%
\bibitem{zh} P. Zhidkov, {\it KdV and Nonlinear Schr\"odinger equations : qualitative theory}, Lecture notes in Mathematics 1756, Springer, 2001.

\end{thebibliography}
\end{document}